\documentclass[11pt,reqno]{amsart}

 \usepackage{amsfonts,amsmath,amscd,amssymb,amsbsy,amsthm,amstext,amsopn,amsxtra}
 \usepackage{color,fullpage,mathrsfs}
 \usepackage{subfigure}
 \usepackage{verbatim} 
 \usepackage{stmaryrd}
 \usepackage{extarrows}
 \usepackage[all]{xy}
 \usepackage{bm}   
 \usepackage{hyperref}
 \usepackage{tikz}
 \usepackage{tikz-cd}
 \usepackage{enumitem}

 \usetikzlibrary{calc}
 \usetikzlibrary{cd}

 \hypersetup{colorlinks,linkcolor=[rgb]{0,0,1},citecolor=[rgb]{1,0,0}}
 
\usepackage{cancel}

 \numberwithin{equation}{section}

\let\nc\newcommand
\let\renc\renewcommand

\theoremstyle{plain}
\newtheorem{thm}{Theorem}
\newtheorem{prop}[thm]{Proposition}
\newtheorem{cor}[thm]{Corollary}
\newtheorem{lem}[thm]{Lemma}
\newtheorem{conjecture}[thm]{Conjecture}

\theoremstyle{definition}
\newtheorem{defn}[thm]{Definition}

\theoremstyle{remark}
\newtheorem{remark}[thm]{Remark}
\newtheorem{rem}[thm]{Remark}

\numberwithin{thm}{section}

\makeatletter
\renewcommand{\subsection}{\@startsection{subsection}{2}{0pt}{-3ex
plus -1ex minus -0.2ex}{-2mm plus -0pt minus
-2pt}{\normalfont\bfseries}} \makeatother

\numberwithin{equation}{section} 

\DeclareMathOperator{\gr}{\mathrm{gr}}

\DeclareMathOperator{\Rep}{\mathrm{Rep}}

\newcommand{\bv}{\boldsymbol{v}}

\newcommand{\beq}{\begin{equation}\label}
\newcommand{\eeq}{\end{equation}}

\DeclareMathOperator{\Spec}{\mathrm{Spec}}

\newcommand{\iso}{{\;\stackrel{_\sim}{\longrightarrow}\;}}

\DeclareMathOperator{\Hom}{\mathrm{Hom}}
\DeclareMathOperator{\GL}{\mathrm{GL}}

\nc{\Z}{\mathbb{Z}}
\newcommand{\N}{\mathbb{N}}
\newcommand{\Q}{\mathbb{Q}}
\newcommand{\R}{\mathbb{R}}
\newcommand{\C}{\mathbb{C}}

\nc{\rank}{\textrm{rank} \,}
\nc{\ds}{\dots}

\let\mc\mathcal
\let\mf\mathfrak

\nc{\mbf}{\mathbf}
\nc{\Res}{\mathsf{Res} \, }
\nc{\Ind}{\mathsf{Ind} \, }
\nc{\cont}{\textrm{cont}}

\newcommand{\Id}{\mathrm{Id}}
\nc{\msf}{\mathsf}

\nc{\minusone}{-1}
\nc{\minustwo}{-2}
\nc{\Mod}{\mathrm{Mod} \,}
\nc{\ms}{\mathscr}
\nc{\Frac}{\mathrm{Frac} \,}
\nc{\ra}{\rightarrow}
\nc{\hra}{\hookrightarrow}
\nc{\lab}{\label}
\renc{\O}{\mc{O}}
\nc{\Tan}{\mc{T}}
\nc{\ul}{\underline}
\nc{\s}{\mathfrak{S}}
\nc{\g}{\mf{g}}
\nc{\pa}{\partial}
\nc{\tit}{\textit}
\nc{\Maxspec}{\mathrm{Maxspec} \, }
\nc{\gldim}{\mathrm{gl.dim}}
\nc{\rkm}{\mathrm{rk} \, (\mf{m})}
\nc{\sm}{\mathrm{sm}}
\nc{\PD}{\mathbb{PD}}
\nc{\hilb}{\textrm{Hilb}}
\nc{\<}{\langle}
\renc{\>}{\rangle}
\nc{\T}{\mathbb{T}}
\nc{\X}{\mathbb{X}}
\nc{\F}{\mathbb{F}}
\nc{\id}{\msf{id}}
\nc{\A}{\mathbb{A}}
\nc{\Grat}{\mc{Grat}}
\nc{\Squo}[1]{\A^{(#1)}}
\nc{\twist}{\mathrm{twist}}
\nc{\Cd}{\mc{C}}
\nc{\Span}{\mathrm{Span}}
\nc{\Grass}{\mathrm{Gr}}
\nc{\Fr}{\mathrm{Fr}}
\nc{\pco}[1]{k[V]^{p\mathrm{co} #1}}
\nc{\Irr}{\mathrm{Irr}}
\renc{\o}{\otimes}
\renc{\gr}{\mathsf{gr}}
\nc{\fin}{\mathrm{fin}}
\nc{\aff}{\mathrm{aff}}
\nc{\algD}{\mf{D}}
\nc{\hr}{\mf{h}_{\textrm{reg}}}
\nc{\D}{\mathscr{D}}
\nc{\PIdeg}{\mathrm{P.I.-degree}}
\nc{\ch}{\mathrm{ch}}
\nc{\ev}{\mathsf{ev}}
\nc{\Stab}{\mathrm{Stab}}
\nc{\Der}{\mathrm{Der}}
\nc{\rightsim}{\stackrel{\sim}{\longrightarrow}}
\nc{\HZ}{H_{\mbf{h},\Z}(\Z_m)}
\nc{\sing}{\mathrm{sing}}
\nc{\dd}{\mathscr{D}}
\nc{\bc}{\mathbf{c}}
\nc{\vc}{\underline{\mathbf{c}}}
\nc{\ba}{\mathbf{a}}
\nc{\reg}{\mathrm{reg}}
\nc{\Amp}{\mathrm{Amp}}
\nc{\Nef}{\mathrm{Nef}}
\nc{\SL}{\mathrm{SL}}
\nc{\Sp}{\mathrm{Sp}}
\nc{\Sym}{\mathrm{Sym}}
\nc{\Mov}{\mathrm{Mov}}
\nc{\Pic}{\mathrm{Pic}}
\nc{\Cs}{\C^{\times}}
\nc{\Nak}[3]{\mf{M}_{{#1}} ({#2},{#3}) }
\nc{\Naka}[2]{\mf{M}({#1},{#2}) }
\nc{\Mtheta}[1]{\mc{M}_{#1}}

\nc{\bw}{\mathbf{w}}

\nc{\bn}{\mathbf{n}}
\nc{\CB}{\mathrm{CB}}
\nc{\GVect}{\Lambda}
\nc{\pZ}{\overline{Z}}
\nc{\Qu}{Q}
\nc{\Supp}{\mathrm{Supp}}
\nc{\mr}{\mathrm}
\DeclareMathOperator{\Ab}{Ab}

\newcommand{\Mq}{\mathfrak{M}}
\newcommand{\dQ}{\overline{Q}}

\newcommand{\CC}{\mathbb{C}}

\newcommand{\RR}{\mathbb{R}}
\newcommand{\ZZ}{\mathbb{Z}}
\newcommand{\NN}{\mathbb{N}}

\DeclareMathOperator{\Aut}{Aut}
\DeclareMathOperator{\Out}{Out}
\DeclareMathOperator{\Inn}{Inn}
\DeclareMathOperator{\Ort}{O}
\DeclareMathOperator{\SO}{SO}

\newcommand{\onto}{\twoheadrightarrow}
\newcommand{\into}{\hookrightarrow}

\newcommand{\one}{\ensuremath{(\mathrm{i})}}
\newcommand{\two}{\ensuremath{(\mathrm{ii})}}
\newcommand{\three}{\ensuremath{(\mathrm{iii})}}

\newcommand{\git}{\ensuremath{/\!\!/\!}}

\nc{\Hyp}{X}
\nc{\U}{U}

\begin{document}

\title{All 81 crepant resolutions of a finite quotient singularity are hyperpolygon spaces}

\author{Gwyn Bellamy}
\address{School of Mathematics and Statistics, 
University of Glasgow, University Place,
Glasgow G12 8QQ, United Kingdom}
\email{gwyn.bellamy@glasgow.ac.uk}
\urladdr{http://www.maths.gla.ac.uk/~gbellamy/}

\author{Alastair Craw} 
\address{Department of Mathematical Sciences, 
University of Bath, 
Claverton Down, 
Bath BA2 7AY, 
United Kingdom}
\email{a.craw@bath.ac.uk}
\urladdr{http://people.bath.ac.uk/ac886/}

\author{Steven Rayan}
\address{Department of Mathematics \& Statistics and Centre for Quantum Topology and Its Applications (quanTA), University of Saskatchewan, 
McLean Hall, 106 Wiggins Road, 
Saskatoon, SK S7N 5E6, 
Canada}
\email{rayan@math.usask.ca}
\urladdr{https://researchers.usask.ca/steven-rayan/}

\author{Travis Schedler} 
\address{Imperial College London, Huxley Building,
South Kensington Campus, London SW7 2AZ,
United Kingdom}
\email{t.schedler@imperial.ac.uk}
\urladdr{https://www.imperial.ac.uk/people/t.schedler}

\author{Hartmut Weiss}
\address{Mathematisches Seminar, Christian-Albrechts-Universit\"{a}t Kiel, Heinrich-Hecht-Platz 6, Kiel D-24118, Germany}
\email{weiss@math.uni-kiel.de}
\urladdr{https://www.math.uni-kiel.de/geometrie/de/hartmut-weiss}

\maketitle

\begin{abstract}
We demonstrate that the linear quotient singularity for the exceptional subgroup $G$ in $\Sp(4,\CC)$ of order 32 is isomorphic to an affine quiver variety for a 5-pointed star-shaped quiver. This allows us to construct uniformly all 81 projective crepant resolutions of $\C^4/G$ as hyperpolygon spaces by variation of GIT quotient, and we describe both the movable cone and the Namikawa Weyl group action via an explicit hyperplane arrangement. More generally, for the $n$-pointed star shaped quiver, we describe completely the birational geometry for the corresponding hyperpolygon spaces in dimension $2n-6$; for example, we show that there are 1684 projective crepant resolutions when $n=6$. We also prove that the resulting affine cones are \emph{not} quotient singularities for $n \geq 6$.
\end{abstract}

\section{Introduction}
This paper has three main goals: (a) to establish a remarkable coincidence of four-dimensional symplectic cones, namely a finite quotient singularity and a hyperpolyon space; (b) to describe completely the birational geometry of hyperpolygon spaces in \emph{all} even dimensions, including all their projective crepant resolutions; and (c) to show that the phenomenon from (a) does not occur for
hyperpolygon spaces in dimension greater than four. 

Goal (a) concerns, on one hand, the symplectic quotient singularity $\C^4/G$ associated to an exceptional subgroup $G$ in $\Sp(4,\CC)$ of order 32 that is known to admit 81 projective crepant resolutions \cite{BellamyCounting,81Coxres}. On the other hand, we consider the affine hyperpolygon space $X_5(0)$ that  admits projective crepant resolutions by variation of GIT quotient for quiver varieties \cite{NakDuke98,KonnoHyper}. 

Goal (b) is achieved by exploiting the identification between hyperpolygon spaces $X_n(0)$ and quiver varieties for a star-shaped quiver. This enables us to give a complete, uniform description of the birational geometry of hyperpolygon spaces $X_n(\theta)$ for any $n\geq 3$, and in the case $n=5$, we use the semi-invariant ring of the quiver to identify $\C^4/G$ with $X_5(0)$.  As a consequence, we describe $\CC^4 / G$ and its crepant resolutions in terms of quiver varieties, leading to a transparent, uniform construction of all $81$ projective crepant resolutions of $\C^4/G$. This provides a conceptual, quiver-theoretic explanation for the main results of \cite{81Coxres} that bypasses computer calculations.

Note that $X_4(0)$ is well known to be isomorphic to the du Val singularity $\C^2/Q_8$ for $Q_8$ the quaternionic group of order eight. As a result, $X_n(0)$ is a finite quotient singularity for $n \leq 5$.  It is therefore an obvious question to ask whether any of these are finite quotient singularities for larger $n$. Goal (c) 
answers this question in the negative: $X_n(0)$ is not a finite quotient singularity for any
$n>5$, or equivalently, for hyperpolygon spaces in dimension greater than four.

The phenomenon that we reveal in goal (a) can be viewed as special in two ways: it establishes one more finite quotient singularity which is a quiver variety; and it reveals another quiver variety which is a finite quotient singularity. 

This is very unusual since,
unlike the case of quiver varieties, \emph{very few} finite quotient singularities admit crepant resolutions.  In fact, the classification of finite, symplectic quotient
singularities admitting a symplectic resolution is nearly complete (there are up to 45 remaining finite symplectic quotients in dimension four that are not expected to admit crepant resolutions, see \cite{BellSchmittThiel}): they fall into one infinite family, namely symmetric powers of du Val singularities, and two exceptional cases in dimension four, one of which is the quotient singularity $\CC^4/G$ that we study in goal (a).  The infinite family is well known to be isomorphic to quiver varieties for extended Dynkin quivers, with crepant resolutions given by variation of GIT quotient (after framing the quiver). In this context, then, our results for $\CC^4/G$ are consistent with the behaviour in the infinite family.  Note that, in the remaining
exceptional case, although it is unknown whether the resolution of singularities 
 can be constructed by variation of GIT quotient, the singularity itself was recently identified to be a symplectic quotient of a vector space by a reductive group \cite[Example 5.6]{BLLT}, also known as a Higgs branch variety.

Thus a consequence of our result is that \emph{all} finite symplectic quotient singularities admitting crepant resolutions are such quotients (up to the 45 remaining cases mentioned above). The four-dimensional example of this paper, like the symmetric powers of du Val singularities, is a quiver variety
with resolution given by variation of GIT.  In this sense the phenomenon established in goal (a) is not just a special case, but completes a general pattern.

\subsection{An exceptional quotient singularity}
The group $G = Q_8 \times_{\ZZ_2} D_8 < \Sp(4,\CC)$ of order $32$ was identified in \cite{BS-sra} as one for which the linear quotient
$\CC^4/G$ admits a crepant (equivalently, a symplectic) 
resolution of singularities; here
$Q_8 < \Sp(2,\CC)$ is the quaternionic group of order $8$,
$D_8 < \Ort(2,\RR)$ is the dihedral group of order $8$, and $\ZZ_2$ is identified
with their respective centres. 
Subsequently, it was established in
\cite{BellamyCounting} that $\CC^4/G$ admits precisely 81
projective crepant resolutions; simultaneously,
 these resolutions were all constructed explicitly in \cite{81Coxres} by variation of GIT quotient.

More recently,  Mekareeya~\cite{Mek-physqt} discovered that this quotient singularity should be isomorphic to a
Higgs branch variety. More precisely, \cite{Mek-physqt} showed that the Hilbert series of the ring of functions on the singularity $\CC^4/G$ coincides with that of the conical quiver
variety $\Mq_{0}(\bv,0)$ associated to a star-shaped quiver $Q$ with a central
vertex 
and five external vertices, 
where the components of the dimension vector $\bv$ equal 2 on the central vertex
and 1 elsewhere 
(see Section \ref{ss:quiver} for precise definitions). The present work was motivated by our desire to  give a conceptual geometric proof of Mekareeya's numerical observation:

\begin{thm}\label{t:isom}
	There is a Poisson isomorphism $\CC^4 / G \iso X_5(0):= \Mq_{0}(\bv,0)$.
\end{thm}

 While comparing these two varieties directly seems to be difficult, it is natural to compare instead the Cox ring of the singularity $\CC^4 / G $, as studied in \cite{81Coxres, HausenKeicher}, with a semi-invariant ring $\CC[\mu^{-1}(0)]^{\SL_2}$,  where $\mu$ is the moment map arising in the description of $\Mq_{0}(\bv,0)$ as a GIT quotient. We prove that a $(\CC^\times)^5$-invariant subring of $\CC[\mu^{-1}(0)]^{\SL_2}$ is isomorphic to the coordinate ring of $\Mq_{0}(\bv,0)$, and Theorem \ref{t:isom} follows by passing to the torus-invariant subrings.  This semi-invariant ring is closely related to the Cox ring of a crepant resolution of $\Mq_{0}(\bv,0)$; the relationship between these two rings will be discussed in the more general context of quiver varieties in \cite{BCSCox}. 
 
 \subsection{A hyperplane arrangement}
 It turns out that studying quiver varieties for the star-shaped quiver with $n$ external vertices (for $n\geq 3$, see Figure~\ref{fig:quivers}) and dimension vector $\bv=(2,1,1,\dots ,1)$ is not substantially harder than the case when $n=5$. We consider this more general setting, where for any stability parameter $\theta\in \Theta:=\Q^n$,  the quiver variety
 \[
 X_n(\theta):= \Mq_{0}(\bv,\theta) = \mu^{-1}(0)^{\theta\text{-ss}}\git\GL_{\bv}
 \]
 is obtained as a GIT quotient (see Section~\ref{ss:quiver}). For any sufficiently general stability parameter $\theta\in \Theta$ satisfying $\theta_i>0$, this \emph{hyperpolygon space} provides a hyperk\"ahler analogue of a moduli space of polygons with prescribed edge lengths \cite{HausmannKnutson}. Hyperpolygon spaces have proved to be an effective testing ground for conjectures on quiver varieties; see, for example, \cite{KonnoHyper,HaradaProudfoot05}.

 To study these spaces, consider the following hyperplane arrangement in $\Theta$:
\[
   \mc{A} = 
   \big\{ L_i, H_I \mid 1\leq i\leq n, \{1\}\subseteq I\subseteq \{1,\dots,n\}\big\},
\]
where $L_i = \{ \theta \in \Theta \mid \theta_i = 0\}$ is a coordinate hyperplane and $H_I= \{ \theta \in \Theta \mid \sum_{i \in I} \theta_i = \sum_{j \notin I} \theta_j\}$. This  arrangement is 
known in the theory of (hyper-)polygon spaces: avoiding the hyperplanes $H_I$ is precisely the condition that your polygon should not contain a line \cite[Proposition~4.3]{HausmannKnutson}. Our interest lies largely with stability parameters lying in the complement to these hyperplanes: a \textit{chamber} of $\Theta^{\reg} = \Theta \smallsetminus \bigcup_{H\in \mathcal{A}} H$ is the intersection with $\Theta$ of a connected component of the real hyperplane arrangement complement in $\Theta\otimes_{\ZZ} \R$. We prove (see Proposition~\ref{prop:Hnthetasmooth}) that $\Hyp_n(\theta)$ is smooth if and only if $\theta \in \Theta^{\reg}$ and, moreover, the chambers in the complement of $\mathcal{A}$ are precisely the GIT chambers arising in the GIT construction of  the spaces $X_n(\theta)$.

 \subsection{The birational geometry of hyperpolygon spaces}
 Obtaining an explicit understanding of the GIT chamber decomposition proved to be a key step in the recent quiver variety description of all projective crepant resolutions for symmetric powers of Du Val singularities \cite{BC}. Here, we apply similar arguments to show that variation of GIT quotient produces all projective crepant resolutions of the affine hyperpolygon space $X_n(0)=\Mq_0(\bv,0)$. The resolutions $X_n(\theta)$ for $\theta\in \Theta^{\reg}$ all share the same movable cone of line bundle classes, with different resolutions corresponding to different ample cones. We use the linearisation map for the GIT construction to identify each GIT chamber in $\Theta$ with the ample cone of the appropriate hyperpolygon space (up to the action of Namikawa's Weyl group, in this case $\ZZ_2^n$). We also show that every projective crepant resolution arises in this way, giving us a precise count of the resolutions.  

To state the results in more detail, observe that the positive orthant 
\[
F = \{ \theta \in \Theta \mid \theta_i \ge 0, \, \forall i \neq 0\}
\]
is the union of the closures of a collection of GIT chambers in $\Theta$. In \eqref{eqn:C_+} we fix once and for all a chamber $C_+$ in $F$, 
define $X:=X_n(\theta)$ for $\theta\in C_+$ and write $Y:=X_n(0)$. Let $N^1(X/Y)$ denote the relative N\'{e}ron--Severi space, defined as the rational vector space spanned by equivalence classes of line bundles on $X$ up to numerical equivalence relative to curves contracted to a point in $Y$.

\begin{thm}
\label{thm:resolutionsofstar}
Let $n\geq 5$. There is an isomorphism $L_F\colon \Theta\rightarrow N^1(X/Y)$ of rational vector spaces such that:
 \begin{enumerate}
     \item[\one] $L_F$ identifies the closed polyhedral cone $F$ with the movable cone $\Mov(X/Y)$; and 
     \item[\two] $L_F$ induces a bijection between the chambers in $F$ and the ample cones $\mr{Amp}(X_i/Y)$ of all projective crepant resolutions of $Y$.
\end{enumerate}
 Thus, for each $\theta\in F$, the space $X_n(\theta)$ is isomorphic to the birational model of $X$ determined by the line bundle $L_F(\theta)$, and conversely, every partial crepant resolution of $Y=X_n(0)$ is of the form $X_n(\theta)$ for some $\theta\in F$.
\end{thm}

In the course of proving Theorem~\ref{thm:resolutionsofstar}, we show that crossing any wall contained in the interior of $F$ induces a Mukai flop between the corresponding hyperpolygon spaces. This result was originally established by Godinho--Mandini~\cite[Section~4]{GodinhoMandini}; here, we bypass the study of wall crossing for moduli spaces of parabolic Higgs bundles from~\cite{Thaddeus02}, deducing the result instead from an \'{e}tale local description of wall crossing for quiver varieties \cite[Section~3]{BC}.

\subsection{Beyond the movable cone}
There is an action of the group $\Z_2^n$ on $\Theta$, where the generator of the $i^{\text{th}}$ factor acts by reflection in the supporting hyperplane $\{\theta\in \Theta \mid \theta_i=0\}$ of the cone $F$. Notice that $\Z_2^n$ permutes the hyperplanes in $\mc{A}$, so it acts on the set of all chambers. 

\begin{prop}
\label{prop:beyondF}
 Let $n\geq 5$. The isomorphism $L_F$ from Theorem~\ref{thm:resolutionsofstar} identifies the action of $\Z_2^n$ on $\Theta$ with the action of the Namikawa Weyl group on $N^1(X/Y)$. In particular, $\theta \in C$ and $\theta' \in C'$ satisfy $\Hyp_n(\theta) \cong \Hyp_n({\theta'})$ as spaces over $\Hyp_n(0)$ if and only if there exists $w \in \Z_2^n$ such that $w(C) = C'$. 
\end{prop}

The spaces $X_n(\theta)$ have traditionally been studied only for those $\theta\in \Theta^{\reg}$ satisfying $\theta_i>0$ for $1\leq i\leq n$, i.e.\ for parameters in the interior of $F$ \cite{KonnoHyper,HaradaProudfoot05, GodinhoMandini}, in which case the values of the components $\theta_i$ record the edge-lengths of polygons in the class parametrised by an irreducible component in the core of $X_n(\theta)$. Proposition~\ref{prop:beyondF} extends this observation to all $\theta\in \Theta^{\reg}$, because we have
$\Hyp_n(\theta) \cong \Hyp_n(w(\theta))$ for the unique $w\in \ZZ^n_2$ satisfying $w(\theta)\in F$. In this case, the absolute values $\vert \theta_i\vert = w(\theta)_i$ for $1\leq i\leq n$ record the appropriate edge-lengths. In short, the space $X_n(\theta)$ can legitimately be described as a \emph{hyperpolygon space} for any $\theta\in \Theta^{\reg}$. 

 Theorem~\ref{thm:resolutionsofstar}\two\ establishes the following result.
 
\begin{cor}
\label{cor:count}
 For $n\geq 5$,  the number of projective crepant resolutions of $\Hyp_n(0)$ is equal to the number of chambers in $F$. 
\end{cor}

 Combining this result with Theorems~\ref{t:isom} and \ref{thm:resolutionsofstar} when $n = 5$ recovers all $81$ projective crepant resolutions of $\C^4/G$ that were constructed explicitly by Donten-Bury--Wi\'{s}niewski~\cite{81Coxres}, but here we demonstrate in addition that they are all hyperpolygon spaces. In this case, the hyperplanes in $\mathcal{A}$ are those appearing in \cite[Theorem 4.2.1]{BS-sra}. When $n = 6$, we compute that there are precisely $1684$ resolutions. After the first draft of this paper appeared, Alastair King observed that the sequence recording the number of chambers in $F$ for $n\geq 4$ appears in the OEIS~\cite{OEIS} and, in particular, the number of projective crepant resolutions of $X_n(0)$ for small values of $n$ is as follows (see Remark~\ref{rem:adk}):  
\begin{center}
\begin{tabular}{c||c|c|c|c|c}
    $n$ &  5 & 6 & 7 & 8 & 9 \\ \hline 
    resolution count & 81 & 1,684 &  122,921 & 33,207,256 & 34,448,225,389.
\end{tabular}
\end{center}
 In general, counting the chambers in $F$ for large values of $n$ appears to be computationally difficult.

\subsection{Quotient singularities and Hamiltonian reduction.}
 
It is a classical fact that, for a finite subgroup $\Gamma \subset \SL(2,\C)$, the symplectic quotient singularity $\C^{2n} / (\s_n \wr \Gamma)$ can be realised as a quiver variety, and it was shown in \cite{BC} that all projective crepant resolutions of this quotient are obtained as quiver varieties by variation of GIT quotient. Therefore, it is natural to ask if every symplectic quotient singularity that admits a projective crepant resolution can be realised as a quiver variety (or, more generally, as a Hamiltonian reduction by a connected group). The only known example not already mentioned is the quotient $\C^4/G_4$ that was studied in \cite{Singular} and \cite{LehnSorger}. Though the singularity can be described as a Hamiltonian reduction \cite[Example~5.6]{BLLT}, we do not know of a construction of this quotient as a quiver variety. 

Motivated by the behaviour of the hyperk\"ahler metrics on $X_n(\theta)$ (see below) it is natural to ask, conversely, if $X_{n}(0)$ is isomorphic to a symplectic quotient singularity for some $n > 5$. Our final main result answers this question:

\begin{thm}\label{thm:notisoquotientsing}
For $n > 5$, the space $X_n(0)$ is not isomorphic to a symplectic quotient singularity.
\end{thm}

The proof depends on the (incomplete) classification of symplectic quotient singularities admitting a symplectic resolution. 

\subsection{The hyperk\"ahler geometry of hyperpolygon spaces}
By virtue of the construction of the hyperpolygon space $X_n(\theta)$ as a hyperk\"ahler quotient, in the sense of \cite{HKLR}, it carries a complete hyperk\"ahler metric for any generic choice of $\theta$. Recently, there has been much interest in the asymptotic geometry of such spaces. The resolution of singularities $X_n(\theta) \to X_n(0)$ identifies the singularity $X_n(0)$ with the tangent cone at infinity of the hyperk\"ahler manifold $X_n(\theta)$. The full description of the asymptotic geometry of $X_n(\theta)$ requires precise estimates on the decay of the true hyperk\"ahler metric on $X_n(\theta)$ to the (singular) cone metric on $X_n(0)$.

It is known that $X_{4}(\theta)$ is asymptotically locally Euclidean (ALE). As a further generalisation of ALE, a smooth noncompact hyperk\"ahler manifold is  called \emph{quasi-asymptotically locally Euclidean} (QALE) if its asymptotic geometry is modelled on $\C^{2k}/G$ for a finite subgroup $G$ in $\Sp(k)$. Both ALE and QALE metrics have maximal, that is, Euclidean volume growth. Theorem~\ref{t:isom} shows that the hyperpolygon space $X_{5}(0)$ is the orbifold cone $\CC^4/G$ for the group $G = Q_8 \times_{\ZZ_2} D_8$. For generic $\theta$, we expect that $X_{5}(\theta)$ is quasi-asymptotically locally Euclidean; see Conjecture~\ref{conj:QALE}. 

We prove that such a result does \emph{not} hold for polygons of more than five sides:

\begin{cor}\label{cor:QALEnot}
  For $n > 5$, $X_n(\theta)$ is not quasi-asymptotically locally Euclidean for any $\theta$.  
\end{cor}

We now provide an outline of the paper, making reference to goals (a), (b) and (c) introduced above. Section~\ref{sec:2} provides the combinatorial framework that allows us to study hyperpolygon spaces and the associated GIT wall-and-chamber structure. Section~\ref{sec:theresofhyperpoly} achieves goal (b) by describing the birational geometry of the hyperpolygon spaces $X_n(\theta)$ for $n\geq 5$ (in Theorem~\ref{thm:movable}). In Section~\ref{sec:4}, we specialise to the case $n=5$ to achieve goal (a) (see Theorem~\ref{thm:4foldsing}) by marrying our description of the $\SL_2$-invariant ring (see Proposition~\ref{p:gens-repsl2}) with the Cox ring $\CC[V]^{[G,G]}$ of the quotient singularity $\CC^4/G$ (see Proposition~\ref{p:rels-qt}). Section~\ref{sec:5} achieves goal (c) by showing that $X_n(0)$ is not isomorphic to a symplectic quotient singularity for any $n>5$; the key ingredient here describes the movable cone of a product (see Proposition~\ref{prop:movableproduct}) and hence reduces the problem to the study of irreducible symplectic reflections groups. Finally, Section~\ref{sec:6} outlines the implications of our results for metrics on hyperpolygon spaces, leading to a proof of Corollary~\ref{cor:QALEnot}.

\medskip

\noindent \textbf{Acknowledgements.\ }  We thank Amihay Hanany for pointing out \cite{Mek-physqt} and for multiple discussions, and Alastair King for the observation that forms Remark~\ref{rem:adk}. We also thank the referees for their comments and corrections. SR thanks Hiraku Nakajima and Laura Schaposnik for helpful conversations.  SR and HW thank Laura Fredrickson, Rafe Mazzeo, and Jan Swoboda for useful discussions. GB and AC were partially supported by a Research Project Grant from the Leverhulme Trust. SR acknowledges the support of an Natural Sciences and Engineering Research Council of Canada (NSERC) Discovery Grant.  HW acknowledges the support of the Deutsche Forschungsgemeinschaft (DFG) within SPP 2026 ``Geometry at infinity''.

\section{Hyperpolygon spaces}
\label{sec:2}
In this section we define hyperpolygon spaces as quiver varieties, and we recall some results that hold for any quiver variety. 
We also describe the GIT chamber decomposition for these spaces, and  describe several invariant rings associated to the quiver GIT construction.

\subsection{Hyperpolygon spaces as quiver varieties}
\label{ss:quiver}

The hyperpolygon spaces are the hyperk\"ahler analogues of polygon spaces, first introduced by Konno \cite{KonnoHyper}. 

To construct these spaces, let $n\geq 3$ and let  $Q = Q(n)$ be the star-shaped quiver 
with vertex set
$Q_0 := \{0, 1, \ldots, n\}$ and arrow set
$Q_1 := \{a_1,\ldots, a_n\}$, where for $1\leq i\leq n$, the arrow $a_i$ has tail at $0$ and head at $i$. Let $\dQ$ be the doubled quiver of $Q$, defined by $\dQ_0=Q_0$ and $\dQ_1 = Q_1 \sqcup \{a_1^*, \ldots, a_n^*\}$, where for $1\leq i\leq n$, the arrow $a_i^*$ has tail at $i$ and head at $0$. Consider the dimension vector
$\bv = (2,1,1,\dots,1) \in \NN^{Q_0}$ with $\bv_0=2$
and $\bv_i=1$ for $i \neq 0$. Let $V_0=\CC^2$ and $V_i=\CC$ for $1\leq i\leq n$.
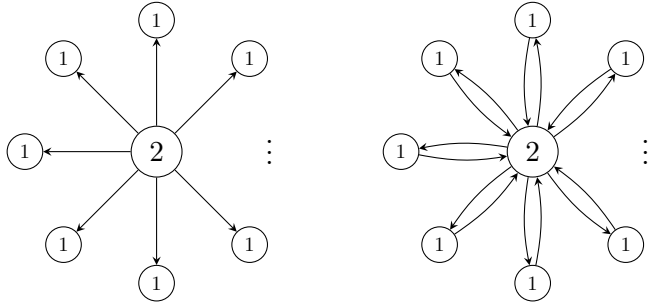
\begin{figure}[ht!]
\begin{tikzpicture}
  \node [circle, draw](0L) at (0,0) {$2$};
  \node [circle, draw](0R) at (5,0) {$2$};
  \foreach \a in {45, 90, ..., 359}
  {
   \node[circle, draw, scale=0.7] (Lv) at ($(0L)+(\a:1.75)$) {$1$};
    \path[-stealth] (0L) edge (Lv);
    \node[circle, draw, scale=0.7] (Rv) at ($(0R)+(\a:1.75)$) {$1$};
    \path[-stealth] (0R) edge[bend right=10] (Rv);
    \path[-stealth] (Rv) edge[bend right=10] (0R);
  };
    \node (Ld) at ($(0L)+(1.5,0.1)$) {$\vdots$};
    \node (Rd) at ($(0R)+(1.5,0.1)$) {$\vdots$};
\end{tikzpicture}
\caption{The quiver $Q$ (left) and the doubled quiver $\overline{Q}$ (right), each with $n+1$ vertices.}
\label{fig:quivers}
\end{figure}

 \noindent The representation space $\Rep(\overline{Q},\bv) := \bigoplus_{i = 1}^n\Hom(V_0, V_i)\oplus \bigoplus_{i = 1}^n\Hom(V_i, V_0)$ admits an action by $\GL_{\bv}:= \GL_2 \times (\CC^\times)^n$, where $g\cdot \big((B_{a_i}),(B_{a_i^*})\big) = \big((g_iB_{a_i}g_0^{-1}), (g_0B_{a_i}g_i^{-1})\big)$ for $g\in \GL_{\bv}$.
 The moment map $\mu\colon \Rep(\overline{Q},\bv) \to
\mathfrak{gl}_{\bv} = \mathfrak{gl}_2 \oplus \mathfrak{gl}_1^n$ for this action takes the form
\[
\mu\big((B_{a_i}),(B_{a_i^*})\big) =  \left(\sum_{j=1}^n B_{a_j^*}B_{a_j}, -B_{a_1}B_{a_1^*}, \dots, -B_{a_n}B_{a_n^*}\right).
\]
Define $\Theta:=\{\theta\in \mathbb{Q}^{n+1} \mid \theta\cdot \bv = 0\}$.  For any $\theta\in \Theta$ and any central $\lambda\in \mathfrak{gl}_{\bv}$, let $\mu^{-1}(\lambda)^{\theta\text{-ss}}$ denote the locus of $\theta$-semistable representations parametrised by points in $\mu^{-1}(\lambda)$. The associated (Nakajima) \emph{quiver variety}~\cite{NakDuke98} is the GIT quotient 
\[
\mathfrak{M}_\lambda(\bv,\theta) = \mu^{-1}(\lambda)\git_\theta \GL_{\bv} := \mu^{-1}(\lambda)^{\theta\text{-ss}}/\GL_{\bv}
\]
 (we use the framing vector $\mathbf{w}=0$). We are particularly interested in the special case $\lambda=0$, and to emphasise the dependence on the number $n$ of external vertices in that case, we write 
 \[
X_n(\theta) := \mathfrak{M}_0(\bv,\theta) =  \mu^{-1}(0)^{\theta\text{-ss}}/\GL_{\bv}
\]
 for the resulting GIT quotient. It is easy to see that the variety $X_n(\theta)$ contains the GIT quotient $M_n(\theta):= \Rep(Q,\bv)\git_\theta \GL_{\bv}$ that is a \emph{polygon space} \cite{HausmannKnutson} for parameters $\theta$ satisfying $\theta_i>0$ for $i>0$. See \cite{KonnoHyper,HaradaProudfoot05, GodinhoMandini,FisherRayan,RayanSchaposnik} 
 for more on polygon and hyperpolygon spaces.

\subsection{Combinatorics of star shaped root systems}
Let $e_0, \dots, e_n$ denote the standard basis of $\RR^{Q_0}$, so $e_i \in \NN^{Q_0}$. The Cartan form $(-,-): \ZZ^{Q_0} \times \ZZ^{Q_0} \to \ZZ$ 
for the quiver $\overline{Q}$ is given by
$$
(\alpha,\beta) = 2 \sum_{i = 0}^n \alpha_i \beta_i - \sum_{j = 1}^n (\alpha_0 \beta_j + \alpha_j \beta_0) 
$$
for $\alpha=\alpha_0e_0+\cdots + \alpha_ne_n$ and $\beta=\beta_0e_0+\cdots + \beta_ne_n$. Define
$$
p(\alpha) = 1 - \frac{1}{2} (\alpha, \alpha) = 1 - \sum_{i = 0}^n \alpha_i^2 + \sum_{j = 1}^n \alpha_0 \alpha_j.
$$

For $i \in Q_0$, the simple reflection $s_i\colon \ZZ^{Q_0} \to \ZZ^{Q_0}$ is the linear map
given by $s_i(\alpha) := \alpha - (\alpha, e_i) e_i$.
The \emph{real roots} are the elements of $\NN^{Q_0}$ obtainable from the $e_i$ by applying arbitrary sequences of simple reflections. A vector $\alpha \in \NN^{Q_0}$ is said to be in the \emph{fundamental region} if $(\alpha, e_i) \leq 0$ for all $i$ and the support (= the collection of vertices $i \in Q_0$ such that $\alpha_i \neq 0$) is connected.  The (positive) \emph{imaginary roots} are the elements of $\NN^{Q_0}$ obtainable from elements of the fundamental region by applying arbitrary sequences of simple reflections. An imaginary root $\alpha$ is \textit{isotropic} if $(\alpha,\alpha) = 0$ and \textit{anisotropic} if $(\alpha,\alpha) < 0$. 

\begin{defn}
	Let $R_{\lambda,\theta}^+$
    denote the set of all positive roots $\alpha$
	such that $\alpha \cdot \lambda = 0 = \alpha \cdot
	\theta$. The set $\Sigma_{\lambda,\theta}$ consists of all
	roots $\alpha \in R_{\lambda,\theta}^+$ such that, for every proper
	decomposition $\alpha = \alpha^{(1)} + \cdots + \alpha^{(m)}$
	with $\alpha^{(i)} \in R_{\lambda,\theta}^+$, we have
	$p(\alpha) > p(\alpha^{(1)}) + \cdots + p(\alpha^{(m)})$.
\end{defn}
In terms of representation theory, the set $\Sigma_{\lambda,\theta}$ consists of the dimension vectors of $\theta$-stable representations of the deformed preprojective algebra $\Pi^{\lambda}$ \cite{BellSchedQuiver,CrawleyBoeveyHolland}.

Recall that $\bv = (2,1,1,\ds,1)$. The simple roots $e_i$ are all, by definition, real. Let $\bv^{(i)} = \bv - e_i$ for $i = 1, \ds, n$. We note that $(\bv,e_0) = 4 - n$ and $(\bv,e_i) = 0$ for $i > 0$, which implies that $\bv$ is in the fundamental region when $n \ge 4$. Moreover, it is an anisotropic root when $n \ge 5$ since $p(\bv) = n - 3$. We also note that $p(\bv - e_0) = p(\bv) + (\bv,e_0) - 1 = 0$.

\begin{lem}\label{lem:ve0realroot}
	The dimension vector $\bv - e_0$ is a real root.
\end{lem}

\begin{proof}
	This is by induction on $n$, and it is convenient to allow any  $n\geq 0$. When $n = 0$, $\bv - e_0= e_0$ is a (real) root. Assume $n > 0$. Notice that for any $i > 0$, in particular for $i = n$, we have $s_i(\bv - e_0) = \bv - e_{0} - e_i$ since $(\bv-e_0,e_i) = 1$. If we define $\bv_n = \bv - e_0$ then this shows that $s_n(\bv_n) = \bv_{n-1}$. By induction, $\bv_{n-1}$ is a real root, so $\bv_n$ is a real root.   
\end{proof}

For each subset $I \subset [1,n]$ define $\bv^I = \bv - \sum_{i \in I} e_i$.  

\begin{lem}\label{lem:rootsstar}
	\begin{enumerate}
		\item[\one] A dimension vector $\alpha \le \bv$ is a root of $Q$ if and only if it has connected support and is not equal to $\bv^I$ for $|I| \ge n-2$. 
		\item[\two] A positive root $\alpha \le \bv$ is (a) anisotropic if $\alpha=\bv^I$ with $n - |I| > 4$ (b) isotropic if $\alpha=\bv^I$ with $n - |I| = 4$ (c) real otherwise.   
	\end{enumerate}
\end{lem}

\begin{proof}
	Part \one. Every dimension vector $\alpha \le \bv$ with connected support is of the form $\bv^I$ or $\bv^I - e_0$, for $I \subset [1,n]$, or $e_j$ for some $j > 0$. We have $p(\bv^I) = n - |I| - 3$. Therefore, the dimension vectors $\bv^I$, where $|I| \ge n-2$, cannot be roots because $p(\bv^I) < 0$ in this case.   
	
	We have explained previously that $\bv^I$ for $|I| < n-2$ is a root, and Lemma~\ref{lem:ve0realroot} implies that $\bv^I - e_0$ is always a real root. Similarly, the $e_j$ are clearly roots.  
	
	Part \two\ follows from the above computations. 
\end{proof}

If $\Sigma_{\theta,\lambda}(\bv) = \{ \alpha \in \Sigma_{\theta,\lambda} \mid \alpha \le \bv \}$, then we deduce from Lemma~\ref{lem:rootsstar} that 
\begin{equation}\label{eq:singma0bv}
    \Sigma_0(\bv) = \{ e_0, \ds, e_n, \bv \} \cup \{ \bv^I \, | \, n - |I| \ge 4 \}.
\end{equation}

\begin{lem}\label{lem:stablegenericrephyperpoly}
For any $\theta \in \Theta$ and $\lambda$ with $\lambda \cdot \bv = 0$, we have that:
\begin{enumerate}
    \item[\one] $\bv\in \Sigma_{\theta,\lambda}$, and hence there exists a $\theta$-stable $\Pi^{\lambda}$-representation of dimension vector $\bv$; and 
    \item[\two] the morphism $\mf{M}_\lambda(\bv,\theta) \to \mf{M}_{\lambda}(\bv,0)$ is a projective, birational Poisson morphism.
\end{enumerate}
In particular, $\Hyp_n(\theta)$ is an irreducible variety of dimension $2p(\bv) = 2n - 6$.
\end{lem}

\begin{proof}
Since $\bv \in \Sigma_0$, it follows that $\bv\in \Sigma_{\theta,\lambda}$ for all $\theta \in \Theta$ and $\lambda$ with $\lambda \cdot \bv = 0$. This means that the canonical decomposition of $\bv$ is $(\bv,1)$ for all $\theta, \lambda$; equivalently, for any $\theta\in \Theta$ there exists a $\theta$-stable $\Pi^{\lambda}$-representation of dimension vector $\bv$. Part (ii) follow from part (i) by \cite[Lemma~2.4]{BellSchedQuiver} and the proof of \cite[Proposition~8.6]{CBnormal}. Finally, since $\bv \in \Sigma_{\theta,0}$, the results  \cite[Theorem~1.2, Lemma~3.21]{BellSchedQuiver} imply that $\Hyp_n(\theta)$ is an irreducible variety of dimension $2p(\bv) = 2n - 6$.
\end{proof}

\subsection{Symplectic leaves}\label{sec:sympleaves}

Since the hyperpolygon space is constructed by Hamiltonian reduction, it is a Poisson variety. In fact, by \cite[Theorem~1.2]{BellSchedQuiver}, it is an example of a \emph{symplectic singularity} in the sense of Beauville \cite{Beauville}. This means that it is a normal variety equipped with a symplectic structure on the smooth locus, whose pullback to any resolution of singularities extends to a regular (globally-defined) two-form.   This condition implies that it has a finite stratification by symplectic leaves \cite{Kaledinsympsingularities}, which are smooth locally closed subvarieties.   Since $Y$ is a cone and its generic symplectic form has positive weight for the contracting $\CC^\times$ action, it is moreover a \emph{conical symplectic singularity}.

By \cite[Proposition~3.6]{BellSchedQuiver}, the symplectic leaves of $\mathfrak{M}_\lambda(\bv,\theta)$ are labelled by the representation types $\tau$ of $\bv$. These are tuples $\tau = (\beta^{(1)},n_1; \ds ; \beta^{(k)},n_k)$, where $\beta^{(i)} \in \Sigma_{\theta,\lambda}, n_i \in \N$ and $\bv = \sum_{i = 1}^k n_i \beta^{(i)}$. Below, we consider the case $\lambda=0$ and $\theta = 0$. For $I \subset [1,n]$ with $|I| < n-2$, let $\mc{L}_I$ be the leaf of type $(\bv^I,1; e_j , 1 : j \in I)$.

\begin{lem}\label{lem:sympleavesH0}
    \begin{enumerate}
        \item[\one] The symplectic leaves 
        of $\Hyp_n(0)$ comprise 
    $\mc{L}_I$ for all subsets 
    $I \subset [1\ds n]$ satisfying 
    $\vert I\vert \leq n-4$,
    together with the leaf $
    \{ 0 \}$. The codimension of $\mc{L}_I$ is $2|I|$. 
        \item[\two] \'Etale locally, a slice to the leaf $\mc{L}_I$ is given by $(\C^2 / \Z_2)^{\times |I|}$. 
    \end{enumerate}
\end{lem}

\begin{proof}
Part (i). It follows from the description of $\Sigma_0(\bv)$ in \eqref{eq:singma0bv} that the representations types are $(\bv,1)$ (the open leaf $\mc{L}_{\emptyset}$), $(\bv^I,1;e_j,1 \colon j \notin I)$ (the codimension $2 |I|$ leaves $\mc{L}_{I}$) and $(e_0,2;e_1,1; \ds )$ (the zero-dimensional leaf). 

Part (ii). Let $\ell := n - 3 - |I|$. Let $Q'$ be the quiver with vertices $f_0$ and $f_j$ for $j \notin I$, two arrows from $f_j$ to $f_0$ and $\ell$ loops at $f_0$. We let $\alpha$ be the dimension vector $(1,1,\ds,1)$ for $Q'$. Since $(\bv^I,e_j) = -2$ and $p(\bv^I) = \ell$, \cite[Corollary~4.10]{CBnormal} says that $\Hyp_n(0)$ is \'etale locally isomorphic at $x \in \mc{L}_I$ to $0$ in the quiver variety $\mf{M}_0(\alpha,0)$ associated to $Q'$. Since $\alpha_0 = 1$, the action of $\GL_\alpha$ factors through $(\CC^\times)^{|I|}$ acting at the external vertices. As a result the quotient factors as a product of  quiver varieties: one for the quiver having only one vertex with $\ell$ loops and dimension $(1)$, and $|I|$ quivers having two vertices and two arrows (in the same direction), with dimension $(1,1)$. This implies that $\mf{M}_0(\alpha,0)$ is isomorphic to $\C^{2\ell} \times (\C^2 / \Z_2)^{\times |I|}$. 
\end{proof}

Lemma~\ref{lem:sympleavesH0} says that the codimension two leaves in $\Hyp_n(0)$ are labelled by the representation type $(\bv^{(i)},1;e_i,1)$; the leaf is $\mc{L}_{ \{ i \} }$. The singularity transverse to the leaf $\mc{L}_{\{ i\} }$ is $\C^2 / \Z_2$. 

\begin{rem}
The closure $\overline{\mc{L}}_I$ is isomorphic to the affine hyperpolygon space $\Hyp_m(0)$, where $m = n - |I|$. This is clear set-theoretically, but can also be shown to hold on the level of Poisson schemes using Proposition~\ref{p:gens-repsl2}. The details are left to the interested reader. 
\end{rem}

\subsection{The hyperplane arrangement}
Identify $\Theta = \{ \theta \in \Q^{n+1} \mid \theta \cdot \bv = 0 \}$ with $\Q^n$ by projection away from the $\theta_0$ component. As in the introduction, consider the hyperplane arrangement
\begin{equation}
\label{eq:hyperpolygonarrangementintro}
   \mc{A} = 
   \big\{ L_i, H_I \mid 1\leq i\leq n, \{1\}\subseteq I\subseteq \{1,\dots,n\}\big\},
\end{equation}
where $L_i = \{ \theta \in \Theta \mid \theta_i = 0\}$ is a coordinate hyperplane and $H_I= \{ \theta \in \Theta \mid \sum_{i \in I} \theta_i = \sum_{j \notin I} \theta_j\}$. It is sometimes convenient to identify $\mathcal{A}$ with the subset of $\mathbb{Q}^n$ comprising the set of points lying in the union of all hyperplanes in $\mathcal{A}$, and we do so without comment from now on. 

\begin{prop}
\label{prop:Hnthetasmooth}
	Let $\theta\in \Theta$. The following statements are equivalent:
	\begin{enumerate}
	    \item[\one] the space  $X_n(\theta)$ is non-singular;
	    \item[\two] every $\theta$-semistable $\Pi$-module of dimension vector $\bv$ is $\theta$-stable; and
	    \item[\three] $\theta$ does not lie in $\mc{A}$.
	\end{enumerate}
	We say $\theta$ is \emph{generic} if it satisfies one, and hence any, of these conditions.
\end{prop}  
\begin{proof}
Since $\bv$ belongs to $\Sigma_0$, the equivalence of \one \ and \two \ follows from \cite[Theorem~1.15]{BellSchedQuiver}, which says that in this case a closed point is singular if and only if it corresponds to a non-stable point. 

Note that, if there is a semistable representation which is not stable, then there is also a polystable one (by taking a filtration of semistable representations and taking the associated graded representation).  This exists if and only if there is a proper decomposition of $\bv$ into a sum of vectors from $\Sigma_{\theta}$. We must show that this is true if and only if $\theta \in \mc{A}$. Assume that $\bv = \beta^{(1)} + \cdots + \beta^{(r)}$ with $r > 1$ and $\beta^{(i)} \in \Sigma_{\theta}$. If there exists $\beta^{(i)}$ with $\beta^{(i)}_0 = 2$ then the other $\beta^{(j)}$ are all of the form $e_k$ for some $k \ge 1$ since their support is connected. In particular, there exists a $k$ such that $\theta \cdot e_k = 0$, giving $\theta \in L_k$. 

Otherwise, we assume that $\beta^{(1)}_0 = \beta^{(2)}_0 = 1$. Then all other $\beta^{(j)}$ are of the go  form $e_k$ for some $k\geq 1$, again since their support is connected. In this case, we can write $\bv = \beta^{(1)} + (\beta^{(2)} + \cdots + \beta^{(r)})$. Then both $\alpha := \beta^{(1)}$ and $\beta := (\beta^{(2)} + \cdots + \beta^{(r)})$ are real roots by Lemma~\ref{lem:rootsstar} since $\beta$ will be connected. The hyperplane in $\Theta$ corresponding to $\theta \cdot \alpha = \theta \cdot \beta = 0$ is of the form $H_I$, for some $I$. Thus, $\theta \in \mc{A}$.

Conversely, if $\theta \in \mc{A}$ then one can check, by considering the cases $\theta \in L_k$ and $\theta \in H_I$ separately, that $\bv = \alpha^{(1)} + \cdots+ \alpha^{(j)}$ admits a decomposition into a (proper) sum of positive roots. Each $\alpha^{(i)}$ can, by definition, be written as a sum of roots in $\Sigma_{\theta}$. It follows that $\bv$ can be decomposed into a proper sum of roots in $\Sigma_{\theta}$. By \cite[Theorem~1.3]{BellSchedQuiver}, this implies that there exists a $\theta$-semistable $\Pi$-module of dimension vector $\bv$ that is not $\theta$-stable.
\end{proof}

\begin{remark}
Proposition~\ref{prop:Hnthetasmooth} shows that the hyperplane arrangement $\mathcal{A}$ determines the GIT chamber decomposition of $\Theta$. Nakajima~\cite[(2.7)]{Nak1994} also introduces a chamber decomposition which, for $\lambda=0$, is determined by the arrangement comprising hyperplanes $\{\theta\in \mathbb{Q}^{n+1} \mid \theta\cdot \bv=\theta\cdot u =0\}$
associated to all vectors $u\in \ZZ^{n+1}$ satisfying $0 < u < \bv$ and $(u,u)\leq 2$. Each hyperplane from $\mathcal{A}$ appears in this arrangement, so the chamber decomposition of Nakajima refines the GIT chamber decomposition. However, these chamber decompositions do not coincide since, for example,  the hyperplane of vectors perpendicular to $u=(2,1,\dots,1,0,0)$ does not lie in $\mathcal{A}$.
\end{remark}

Lemma~\ref{lem:stablegenericrephyperpoly} and Proposition~\ref{prop:Hnthetasmooth} together imply the following result; see \cite{Beauville} for the definition of symplectic resolution, noting here that it is equivalent to the morphism being crepant. 

\begin{cor}\label{cor:symplecticreshyperpoly}
The morphism $f_\theta\colon X_n(\theta)
\rightarrow
X_n(0)$ obtained by variation of GIT quotient is a projective, symplectic resolution if and only if $\theta\in \Theta^{\reg}$.
\end{cor}

\begin{prop}
	For all $\lambda$, we have that:
	\begin{enumerate}
	    \item[\one] $\mu^{-1}(\lambda)$ is a reduced, irreducible, complete intersection; and
	    \item[\two] $\mu^{-1}(\lambda) \git \, \SL_2$ is reduced and irreducible.
	\end{enumerate}

\end{prop}

\begin{proof} 
Since $\bv$ belongs to $\Sigma_{\lambda,0}$, part \one \ is due to Crawley-Boevey \cite{CBmomap}. Part \two\ follows from \one. 
\end{proof}

 \subsection{The semi-invariant ring}
We now introduce a presentation of the ring $\CC[\mu^{-1}(\lambda)]^{\SL_2}$. We refer to Figure~\ref{fig:quivers} for notation of the quiver. Given a vertex $i$, let $V_i = \CC^{\bv_i}$ be the vector space at vertex $i$.
Given an arrow $a$, let $\ev_a: \Rep(\dQ,\bv) \to \Hom(V_{a_s}, V_{a_t})$ be the corresponding function. We fix a symplectic form on $V_0 = \CC^2$ and view $\SL_2 = \Sp(V_0)$. This gives an isomorphism $V_0 \cong V_0^*$ which is $\SL_2$-equivariant. We also fix a trivialisation $V_i \cong \CC$ for each $i$. Put together we can view $\ev_a$ as valued in $V_0$ for every arrow $a$. We can therefore define the $\SL_2$-invariant functions $\varphi_{a,b} := \ev_a \wedge \ev_b: \Rep(\dQ,\bv) \to \CC$.  Note that $\varphi_{a,b}=-\varphi_{b,a}$. 
\begin{prop}\label{p:gens-repsl2}
	\begin{enumerate}
		\item[\one] The ring $\CC[\Rep(\dQ,\bv)]^{\SL_2}$ is generated by the $\varphi_{a,b}$, subject to the relations $\varphi_{a,b} \varphi_{c,d} - \varphi_{a,c} \varphi_{b,d} + \varphi_{a,d} \varphi_{b,c} = 0$ (and $\varphi_{a,b}=-\varphi_{b,a}$).
		\item[\two] The ring $\CC[\mu^{-1}(\lambda)]^{\SL_2}$ is the quotient of $\CC[\Rep(\dQ,\bv)]^{\SL_2}$ by the ideal generated by the functions 
  $\varphi_{a_i,a_i^*}-\lambda_i$ for $1\leq i\leq n$ and $
		\sum_i \varphi_{a_i,b} \varphi_{a_i^*, c} - \lambda_0 \varphi_{b,c}$ for all $b,c \in \dQ_1$.
		\item[\three] The subring $\CC[\mu^{-1}(\lambda)]^{\GL_{\bv}}$ is spanned by those monomials in $\varphi_{a,b}$ in which the total number of occurrences of the index $a_i$ is the same as of $a_i^*$ for all $i$.
	\end{enumerate}
\end{prop}
\begin{proof} 
	For \one, generation is formally a direct consequence of
 first fundamental theorem of invariant theory for $\Sp(V_0)=\SL_2$ by Weyl~\cite{Weyl-ftit}. Indeed, applying this theorem as in \cite[Theorem~8.4]{KraftProcesi} with $p=2n$, $q=0$ and $n=2$ shows that $\CC[\Rep(\dQ,\bv)]^{\SL_2}$ is generated by the determinants $\varphi_{a,b}$ (denoted $[\ev_a,\ev_b]$ in \emph{loc.\ cit.}). The relations are a direct consequence of the second theorem of invariant theory \cite{Weyl-ftit}. Equivalently, the relations can be interpreted as the Pl\"{u}cker relations between the determinants $\varphi_{a,b}$ listed above, or as diagrammatic Pfaffians as in \cite[Theorem 3.7]{Gav}; see also \cite[Theorem 4.12]{Hu-Xiao-BMW} for an ideal generating set.
	
	For \two, the first set of equations is precisely the set defined by the moment map at the external vertices. For the second, the moment map condition at the internal vertex says that $\sum_i \ev_{a_i^*} \otimes \ev_{a_i}$ is the constant function valued at $\lambda_0 \Id_{V_0} \in V_0 \otimes V_0^*$, where we use the symplectic form $\omega$ to identify $\ev_{a_i}$ with $\omega(\ev_{a_i},-)\in V_0^*$ (otherwise replace $\Id_{V_0}$ by the inverse of the symplectic form in $V_0 \otimes V_0$). We pair this element with $\ev_c \otimes \ev_b$ using the symplectic form $\omega$ on the left-hand factor and the natural pairing of $V_0$ and its dual on the right-hand factor to obtain
 \begin{equation}
 \label{eqn:pairings}
\sum_i \omega(\ev_{a_i^*}, \ev_c) \cdot \omega(\ev_{a_i},\ev_b) = \sum_i \varphi_{a_i^*,c} \cdot \varphi_{a_i,b}\quad \text{ and }\quad \lambda_0 \omega(\ev_b,\ev_c) = \lambda_0 \varphi_{b,c},
 \end{equation}
where in the second equality, the symplectic form identifies $\Id_{V_0}$ with $\omega(-,-)\in V_0^*\otimes V_0^*$. The functions in \eqref{eqn:pairings} agree, so we obtain the given equations. 
Conversely, note that an invariant which is in the ideal generated by the coefficients of $\sum_a \ev_a \otimes \ev_{a^*} - \lambda_0 \Id_{V_0}$ is precisely one such that it becomes zero after we successively eliminate pairs of indices $a_1, a_1^*$ via the relation. This will be generated by the given equations together with $\sum \varphi_{a_i,a_i^*} - 2 \lambda_0$, but the latter element is in the ideal generated by the first set of equations since $\lambda \cdot \bv = 0$.
	
	For \three, note that $\GL_{\bv}/ \SL_{2} = (\CC^{\times})^{n+1}$, with the diagonal $\CC^\times$ acting trivially. The elements $\varphi_{a_i, a_j}, \varphi_{a_i, a_j^*}$, and $\varphi_{a_i^*, a_j^*}$ are weight functions of weights $e_i + e_j, e_i - e_j$, and $-e_i - e_j$, respectively. The result follows immediately.
\end{proof}
Note that the equations in \two\ allow us to eliminate generators $\varphi_{a_i, a_i^*}$. Together with skew-symmetry, we can restrict generators to $\varphi_{a_i, a_j}, \varphi_{a_i^*, a_j^*}$ for $i < j$, and $\varphi_{a_i, a_j^*}$ for $i \neq j$.

\begin{remark}
Suppose that $\lambda \neq 0$. Then, the first set of relations in (ii) imply, together with (i), that  $\frac{\varphi_{a_i^*, b}}{\varphi_{a_i, b}}$ does not depend on $b \notin\{a_i, a_i^*\}$. Call this element $u_i$, so that $\varphi_{a_i^*, a_j} = u_i \varphi_{a_i, a_j}$. Then we can rewrite our generators as:
\begin{equation}\label{e:red-gens}
	\varphi_{a_i, a_j} u_i^{\epsilon_1} u_j^{\epsilon_2}, \quad \text{for } \epsilon_1,\epsilon_2 \in \{0,1\}  \text{ and  } 1 \leq i < j \leq n.
\end{equation}
Working over $\CC[u_1, \ldots, u_n]$, this allows us to eliminate all relations involving $a_i^*$, leaving:
\begin{gather}\label{e:red-rels1}
	\varphi_{a, b} \varphi_{c,d} - \varphi_{a,c} \varphi_{b,d} + \varphi_{a,d} \varphi_{b,c}, \\
	\label{e:red-rels2}  \sum_i \varphi_{a_i, b} \varphi_{a_i, c} u_i,
\end{gather}
for all $a,b,c,d \in \{a_1, \ldots, a_n\}$, with the conventions $\varphi_{a_i,a_i} := 0$ and $\varphi_{a_j,a_i} := - \varphi_{a_i, a_j}$ for $j > i$. The $\GL_{\bv}$ invariant functions in these terms are those linear combinations of monomials in which each index $i$ occurs as a subscript of $\varphi$ twice as many times as it does as a subscript of $u$.
\end{remark}

\section{All 81 symplectic resolutions are quiver varieties}
\label{sec:theresofhyperpoly}

The results of \cite{BC} generalise to arbitrary quiver varieties, as explained in \cite{BCS}. Here we introduce the set-up and highlight only the computations that are specific to our example $\Hyp_n(0)$.

\subsection{The birational geometry of hyperpolygon spaces}
Consider the cone
\begin{equation}
    \label{eqn:C_+}
    C_+ = \big\{ \theta \in \Theta \mid \theta_1 > \theta_2+\cdots + \theta_n \textrm{ and } \theta_i > 0 \text{ for all }i>0 \big\}
\end{equation}
and let $\theta\in C_+$. Since $\theta_1 > \theta_2+\cdots + \theta_n$, the inequality $\sum_{i \in I} \theta_i > \sum_{j \notin I} \theta_j$ holds for every set $I$ satisfying $\{1\}\subseteq I\subseteq \{1,\dots n\}$. It follows that every such $\theta$ satisfies a strict inequality determined by each hyperplane in $\mathcal{A}$, so $C_+$ is a GIT chamber. For $\theta_+\in C_+$, write 
\[
f\colon X:=X_n(\theta_+)\longrightarrow Y:=X_n(0)\]
for the corresponding projective symplectic resolution. 

Recall the relative N\'{e}ron--Severi space $N^1(X/Y)$ 
defined in the introduction. For a resolution of singularities $X \to Y$, the exceptional locus is the union of all fibres which are not a single point. 

\begin{prop}
\label{prop:dimN1}
The rational vector space $N^1(X/Y)$ has dimension $n$. Moreover, $Y$ is $\mathbb{Q}$-factorial and the exceptional locus of $f$ is the union of $n$ irreducible divisors.
\end{prop}
\begin{proof}
  We claim that the number of irreducible exceptional divisors is exactly $n$. Consider the symplectic leaves of codimension-two in $Y$. By Lemma \ref{lem:sympleavesH0}, and the discussion thereafter, these are the leaves  $\mathcal{L}_{\{i\}}$ for $1\leq i\leq n$, which locally are transverse to type $A_1$ du Val singularities. The restriction of $f$ to a neighbourhood of each codimension-two singularity is the minimal resolution (isomorphic to the blow up of the singular locus) with exceptional fibres $\mathbb{P}^1$.  Now, let $Y_{\leq 2} \subseteq Y$ denote the union of the open symplectic leaf and the symplectic leaves of codimension two. Since $\mathbb{P}^1$ is irreducible, the exceptional locus of $f$ over $Y_{\leq 2}$ has exactly $n$ irreducible components: the preimages of the leaves $\mathcal{L}_{\{i\}}$.  Finally, since
 $f$ is semismall \cite[Corollary~6.11]{Nak1994}, and $f$ is an isomorphism over the smooth locus of $Y$, it follows that $X \setminus f^{-1}(Y_{\leq 2})$ has codimension at least two.  We have proved the claim.
 
 It is a well-known fact that, for any projective resolution $X \to Y$, the classes of the exceptional divisors are linearly independent in $N^1(X/Y)$: this follows from the Hodge index theorem in the surface case \cite[Lemma~3.40]{KollarMori}, and by induction on dimension using a generic hypersurface section in general; see \cite[Lemma 1.1.1]{NamikawaInduced} for a more detailed argument. As a result, $\dim N^1(X/Y) \geq n$. 
 
 For the opposite inequality,  Konno~\cite[Theorem~5.7]{KonnoHyper} computes the Poincar\'{e} polynomial to be
 \[
 P_t(X) 
 = \frac{(1+t^2)^{n-1}}{1-t^2} -\frac{2^{n-1}t^{2(n-2)}}{1-t^2} + t^{2(n-2)}
 = 1+(n-1)t^2 + t^2 + \cdots ,
 \]
 giving $n = b_2(X)=\dim H^2(X,\mathbb{Q})$. 
 Since the cycle map is an isomorphism \cite[Theorem~7.3.5]{NakJAMS}, we may identify $H^2(X,\mathbb{Z})$ with $A^1(X)$, and hence with $\Pic(X)$ since $X$ is smooth. It follows that $\dim N^1(X/Y) \leq \dim \Pic(X) \otimes_{\ZZ} \mathbb{Q}=n$, so the map $\Pic(X) \otimes_{\ZZ} \mathbb{Q} \twoheadrightarrow N^1(X/Y)$ is an isomorphism.
 
Let $E_1, \ldots, E_n$ be the $n$ irreducible exceptional divisors identified above, and set $U:=X \setminus (\bigcup E_i)$. The standard exact sequence on divisor class groups implies that $\Pic(X)$ surjects onto $\Pic(U)$, with kernel spanned by the $\mc{O}_X(E_i)$. Since we have shown that $\Pic(X) \otimes_{\ZZ} \mathbb{Q}$ is spanned by the $\mc{O}_X(E_i)$, the group $\Pic(U)$ is finite.  The  locus over which $f$ is an isomorphism is an open subset of $U$, whose complement in $U$ has codimension at least two. Therefore its Picard group also identifies with the preceding finite abelian group. The same is therefore true for its image $f(U)$, which is the smooth locus of $Y$. Note that the complement of $f(U)$ in $Y_{\leq 2}$ has codimension at least two, so $Y\setminus f(U)$ has codimension two in $Y$.
Since $Y$ is normal, the Weil divisor class group of $Y$ is isomorphic to the Picard group of $f(U)$, hence of $U$.
As this is finite,  $Y$ is $\mathbb{Q}$-factorial.  It follows that the exceptional locus of any resolution of $Y$ is pure of codimension one \cite[1.40]{Debarre}, hence in this case it is precisely the union of the $n$ exceptional divisors.
 \end{proof}
 
 \begin{remark}
 More generally,  any conical symplectic resolution (i.e.\ $\CC^\times$-equivariant symplectic resolution of a cone) satisfies $\Pic(X) \otimes_{\ZZ} \mathbb{Q} \cong  N^1(X/Y) \cong H^2(X, \mathbb{Q})$.  Namely,  Namikawa has explained that the algebraic and analytic rational Picard groups of $X$ are isomorphic. This implies the statement via the exponential exact sequence and Matsusaka's Theorem that any numerically trivial bundle has torsion first Chern class.
 \end{remark}

Let $C\subset \Theta$ be any chamber. Each $\theta\in C$ is generic by Proposition~\ref{prop:Hnthetasmooth} and the dimension vector $\bv$ is indivisible, so King~\cite[Proposition~5.3]{KingStable} implies $X_n(\theta)=\mathfrak{M}_{0}(\bv,\theta)$ is the fine moduli space of $\theta$-stable $\Pi$-modules of dimension vector $\bv$. It carries a tautological vector bundle $T_n(\theta):=T_0\oplus \bigoplus_{1\leq i\leq n} T_i$ whose fibre over each closed point is the corresponding $\Pi$-module. Note that $\rank(T_0)=2$, $\rank(T_i)=1$ for $1\leq i\leq n$, and we choose our normalisation of $T_n(\theta)$ such that $\det(T_0)\cong \mathcal{O}_{X_n(\theta)}$. When $n = 4$, the bundle $T_4(\theta)$ is tilting, but this is far from true when $n = 5$.

Since $\theta\in \Theta^{\text{reg}}$, the birational map $f_{\theta}^{-1}\circ f\colon X\dashrightarrow X_n(\theta)$ is an isomorphism in codimension-one, so we may identify $N^1(X_n(\theta)/Y)$ with $N^1(X/Y)$ by taking the strict transform. As a result, we may define the \emph{linearisation map} $L_{C}\colon \Theta\to N^1(X/Y)$ to be the $\mathbb{Q}$-linear map given by
 \[
 L_C(\theta) =  
 T_1^{\theta_{1}}\otimes \cdots \otimes 
 T_n^{\theta_n}.
 \]
 For $\theta\in C$, we have $L_C(\theta)\in \Amp(X_n(\theta)/Y)$ and hence $L_C(C)\subseteq \Amp(X_n(\theta)/Y)$. The following result establishes that this inclusion is in fact equality.
 
 \begin{prop}
 \label{prop:ample}
Let $C\subset \Theta$ be a chamber. For $\theta\in C$, the map $L_C$ is an isomorphism of rational vector spaces that identifies $C$ and $\overline{C}$ with $\Amp(X_n(\theta)/Y)$ and $\Nef(X_n(\theta)/Y)$ respectively.
 \end{prop}
 \begin{proof}
 Given Proposition~\ref{prop:Hnthetasmooth} and Proposition~\ref{prop:dimN1}, the result follows as in \cite[Proposition~6.1]{BC}.
 \end{proof}

Consider next the closed cone
 \begin{equation}
     \label{eqn:F}
 F = \{ \theta \in \Theta \mid \theta_i \ge 0, \text{ for all }i\neq 0 \}.
  \end{equation}
 This is the union of closures of a collection of GIT chambers in $\Theta$. The maps $L_C$ and $L_{C^\prime}$ associated to a pair of adjacent chambers do not agree in general, but the next result shows that for $n\geq 5$, they agree when $C, C^\prime$ both lie in $F$. To state the result, recall that $L\in N^1(X/Y)$ is said to be \emph{movable} if the stable base locus of $L$ over $Y$ has codimension at least two in $X$. The relative \emph{movable cone} $\Mov(X/Y)$ is the closure in $N^1(X/Y)$ of the cone generated by all movable divisor classes over $Y$.

 \begin{thm}
 \label{thm:movable}
 Let $n\geq 5$. There is an isomorphism $L_F \colon \Theta \to N^1(X/Y)$ of rational vector spaces such that 
 \begin{enumerate}
     \item[\one] $L_F=L_C$ for each chamber $C\subseteq F$; and 
     \item[\two] $L_F$ identifies $F$ with the movable cone $\Mov(X/Y)$ in such a way that each chamber $C\subseteq F$ is identified with the ample cone $\Amp(X_n(\theta)/Y)$ of the hyperpolygon space $X_n(\theta)$ for $\theta\in C$. 
\end{enumerate}
 In particular, every projective crepant resolution of $Y=X_n(0)$ is a hyperpolygon space $X_n(\theta)$ for some generic $\theta\in F$.
\end{thm}
\begin{proof}
Define $L_F:=L_{C_+}$. Proposition~\ref{prop:ample} shows that $L_F$ identifies $C_+$ with $\Amp(X/Y)$. To study the other chambers in $F$ we analyse how $X_n(\theta)$ and its tautological bundle $T_n(\theta)$ change as $\theta$ crosses (or simply touches) walls of the GIT chamber decomposition using the \'{e}tale-local description from \cite[Theorem~3.2]{BC}. There are two cases.

 First, consider a wall in the interior of $F$. Any such wall is of the form $H_I\cap \overline{C}\cap \overline{C'}$ for some subset $\{1\}\subseteq I\subseteq \{1,\dots,n\}$, where $C, C^\prime\subset F$ are chambers separated by the wall. If $\theta_0 \in H_I$ is generic then the two real roots $\alpha := e_0 + \sum_{i \in I} e_i$ and $\beta := e_0 + \sum_{j \notin I} e_j$  satisfy $\theta_0(\alpha) = \theta_0(\beta) = 0$ and $p(\alpha) = p(\beta) = 0$. By Lemma~\ref{lem:rootsstar}, these are the only positive roots $< \bv$ that pair to zero with $\theta_0$. Therefore, the set $\Sigma_{\theta_0}(\bv)$ equals $\{ \alpha,\beta, \bv\}$. As explained in Section~\ref{sec:sympleaves}, this implies that $\Hyp_{n}(\theta_0)$ contains exactly two symplectic leaves, labelled by the representation types $(\bv,1)$ and $(\alpha,1;\beta,1)$ respectively. Since $\alpha, \beta$ are real, the closed leaf labelled by $(\alpha,1;\beta,1)$ is zero-dimensional, so $\Hyp_n(\theta_0)$ has a 
 single isolated singularity. Since $\alpha_0=\beta_0=1$, we have
 \[
 (\alpha,\beta) = 
 2 + \sum_{i = 1}^n \alpha_i \beta_i - \sum_{j = 1}^n (\beta_j + \alpha_j) = 2 - n = - (n-2).
 \]
 Thanks to \cite[Section~3.2]{BC}, we see that in a \'{e}tale neighbourhood of the singular point, $X_n(\theta_0)$ looks locally like the zero point in an affine quiver variety for the quiver with $2$ vertices and $n-2$ arrows from vertex $0$ to vertex $1$. 
This is precisely the quiver encoding $\overline{\mc{O}}_{\mr{min}}$ in $\mf{sl}(n-2)$, and \emph{ibid}.\ shows that crossing this wall induces a birational map $f_{C, C'}\colon X_n(\theta)\dashrightarrow X_n(\theta')$ for $\theta\in C$ and $\theta'\in C'$. \'{E}tale-locally around the singular point, this is a Mukai flop~\cite[Theorem~0.7]{Mukai84} of the form
$$
\begin{tikzcd}
	T^* \mathbb{P}(V) \ar[dr] \ar[rr,dashed] & & T^* \mathbb{P}(V^*) \ar[dl] \\
	& \overline{\mc{O}}_{\mr{min}} & 
\end{tikzcd}
$$
where $\dim V = n-2 \geq 3$. The morphisms from both $X_n(\theta)$ and $X_n(\theta')$ to $X_n(\theta_0)$ are semi-small, so the unstable locus for each morphism has codimension greater than one. It follows as in the proof of \cite[Corollary~6.3]{BC} that $L_C=L_{C'}$ and, moreover, the birational map $f_{C, C'}$ is a flop. In particular, the linearisation maps $L_C$ agree for all chambers $C$ in $F$, proving part \one. 

Otherwise, the wall lies in the boundary of $F$. Any such wall is $L_i\cap \overline{C}$ for some $1\leq i\leq n$ and chamber $C\subseteq F$. If $\theta\in C$ and $\theta_0$ is generic in the wall, then there is a morphism $X_n(\theta)\to X_n(\theta_0)$. Then $\Hyp_n(\theta_0)$ has two symplectic leaves, labelled by $(\bv,1)$ and $(\bv - e_i,1;e_i,1)$. Theorem~3.3 of \cite{BellSchedQuiver} says that, \'etale locally on the closed leaf labelled by $(\bv - e_i,1;e_i,1)$, the variety $X_n(\theta_0)$ is isomorphic to a neighbourhood of zero in the affine quiver variety $\mf{M}_0((1,1),0)$, where the underlying quiver has $n-4$ loops at one vertex and two arrows from one vertex to the other. That is, locally we get $\C^{2(n-4)} \times \C^2 / \Z_2$. Thus, it follows from \cite[Theorem~3.2]{BC} that the morphism $X_n(\theta)\to X_n(\theta_0)$ is a divisorial contraction, so $L_F(\theta_0) = L_C(\theta_0)$ lies in the boundary of $\Nef(X_n(\theta)/Y)$. This is true for all walls in the boundary of $F$, so $L_F$ identifies the boundary of $F$ with the boundary of $\Mov(X/Y)$.

 It follows that the relative ample cone of every projective crepant resolution $X'\to Y$ is identified via strict transform with a cone in $\Mov(X/Y)$, so $X'=X_n(\theta')$ for $\theta'\in L_F^{-1}(\Amp(X'/Y))$.  
\end{proof}

\begin{remark}
\label{rem:surfaceCase}
For $n=4$, the proof of Theorem~\ref{thm:movable} goes through verbatim, except that the birational map $f_{C,C'} \colon X_n(\theta)\dashrightarrow X_n(\theta')$ obtained by crossing any interior wall in $F$ is an isomorphism because $\dim V=2$. Thus, every wall induces a divisorial contraction $X_n(\theta)\to X_n(\theta_0)$. In this case, the hyperplane arrangement \eqref{eq:hyperpolygonarrangementintro} decomposes $F$ into the union of 12 GIT chamber closures, each of which is isomorphic to the relative Nef cone by Proposition~\ref{prop:ample}, which in turn is the relative movable cone because $\dim X=2$.
\end{remark}

\subsection{Beyond the fundamental domain}
Theorem~\ref{thm:movable} is enough to understand completely the birational geometry of any hyperpolygon space $X_n(\theta)$ for $\theta\in F$. For completeness, we explain how the Namikawa Weyl group allows us to extend this understanding to any hyperpolygon space.

Let $n\geq 5$. The Namikawa Weyl group is generated by reflections in the supporting hyperplanes $L_1, \dots, L_n$ of the cone $F$ from \eqref{eqn:F}, so it is $\mathbb{Z}_2^{n}$ in this case. Note that this action permutes the GIT chambers in $\Theta$, and it is explicitly realised on the hyperpolygon space $X_n(\theta)=\mf{M}_{0}(\bv,\theta)$ by Nakajima's reflection functors since $(\bv,e_i) = 0$ for all $i > 0$. As explained in \cite[Section 7]{BC} and references therein, the reflection $s_i \colon \Theta \to \Theta$ at each external vertex $1\leq i \leq n$ can be lifted to a reflection isomorphism $S_i \colon \Hyp_n(\theta) \to \Hyp_n(s_i(\theta))$ of varieties over $Y$ provided $\theta_i \neq 0$. In particular, for any $\theta \in \Theta^{\reg}$ there is a Poisson isomorphism $S_i \colon \Hyp_n(\theta) \iso \Hyp_n(s_i(\theta))$. The reflections $s_1, \dots, s_n$ generate the Namikawa Weyl group $\mathbb{Z}_2^{n}$, so we obtain the following.

\begin{lem}
\label{lem:reflection}
Let $n\geq 5$.  Let $C, C^\prime\subset \Theta$ be chambers and let $\theta \in C$,  $\theta^\prime \in C^\prime$. Then $X_n(\theta)\cong X_n(\theta^\prime)$ as schemes over $Y$ if and only if there exists $w \in \mathbb{Z}_2^{n}$ such that $w(C) = C^\prime$. In particular, if $C, C'$ are separated by a wall contained in some $L_i$, then $X_n(\theta)\cong X_n(\theta^\prime)$ for $\theta \in C$ and $\theta^\prime \in C^\prime$.
\end{lem}

The proof of the analogous result from \cite[Corollary~7.10]{BC} is more involved because one generator of the Namikawa Weyl group in that case is not a Nakajima reflection.

\begin{remark}
In the proof of Theorem~\ref{thm:movable}, we established that $X_n(\theta)$ undergoes a Mukai flop as $\theta$ crosses any interior wall of $F$. This result is due originally to Godinho--Mandini~\cite[Theorem~4.2]{GodinhoMandini}, but their approach is quite different. They identify $X_n(\theta)$ with a moduli space of stable, rank two, holomorphically trivial parabolic Higgs bundles over $\mathbb{P}^1$ with fixed determinant and trace-free Higgs field, in which case they are able to reduce to previous work on variation of GIT quotient by Thaddeus~\cite{Thaddeus02}.   In fact, Lemma~\ref{lem:reflection} implies more: every wall crossing in $\Theta$ induces either a Mukai flop (for walls in some $H_I$) or an isomorphism of hyperpolygon spaces (for walls in some $L_i$).
\end{remark}

\subsection{Counting crepant resolutions}
The result of Theorem~\ref{thm:movable} allows us to count the number of projective crepant resolutions of the singularity $X_n(0)$ by counting the number of GIT chambers in the fundamental domain $F$ for the action of the Namikawa Weyl group on $\Theta$.

\begin{prop}
The number of projective crepant resolutions of $X_n(0)$ is equal to 1 when $n=4$, it is $81$ when $n=5$, and it is $1684$ when $n=6$.
\end{prop}We remark that this is the only proof in this article that relies on computer calculation.
\begin{proof}
 When $n = 4$, the characteristic polynomial of the arrangement $\mc{A}$ is $q^4 - 12q^3 + 50q^2 - 84q + 45$. Evaluating this polynomial at $-1$ gives $192$, so $\Theta^{\reg}$ contains $192$ chambers. In this case 
 $X_4(0)$ is the Kleinian singularity of type $D_4$. Therefore, the Namikawa Weyl group is the Weyl group of type $D_4$, which has order $2^3 \cdot 4! = 192$. Any chamber provides a fundamental domain for this action, and indeed the minimal resolution of the Kleinian singularity is unique. Note that $F$ contains 12 chambers in this case.

 When $n=5$, the characteristic polynomial of  $\mc{A}$ is $q^5-21q^4 +  170q^3 - 650q^2 + 1125q - 625$. Evaluating this polynomial at $-1$ gives $-2592$, so $\Theta^{\reg}$ contains $2592$ chambers. Taking into account the action of the Namikawa Weyl group $\ZZ_2^{5}$ implies that $F$ contains $2592/32 = 81$ chambers. 

 Finally, when $n = 6$, a brutal computer calculation using Singular shows that characteristic polynomial of the arrangement $\mc{A}$ is $q^6 - 38q^5 + 607q^4 - 5100q^3 + 22935q^2 - 48750q + 30345$. Evaluating this polynomial at $-1$ gives $107776$, so $\Theta^{\reg}$ 
 comprises $107776$ chambers. The Namikawa Weyl group is  $\ZZ_2^{6}$ in this case, so $F$ contains $107776/64 = 1684$ chambers. 
\end{proof}

\begin{rem}
\label{rem:adk}
Alastair King subsequently observed that the sequence $12, 81, 1684,...$ recording the number of chambers in $F$ for $n\geq 4$ coincides with the sequence counting the number of positive self-dual threshold functions of $n$ variables; this sequence for $n\geq 1$ appears in \cite{OEIS}. The known terms in this sequence for $n\geq 4$ are:
\begin{center}
\begin{tabular}{c|c|c|c|c|c|c}
    $n$ &  4 & 5 & 6 & 7 & 8 & 9 \\ \hline 
    chamber count & 12 & 81 & 1,684 &  122,921 & 33,207,256 & 34,448,225,389.
\end{tabular}
\end{center}
For $n\geq 5$, these numbers count the projective crepant resolutions of $X_n(0)$ by Theorem~\ref{thm:movable}. For $n=4$, the (Namikawa) Weyl group is the group generated by reflections in all hyperplanes of $\mc{A}$, whereas the group generated by reflections in the supporting hyperplanes of $F$ is a subgroup of index 12 (see Remark~\ref{rem:surfaceCase}).
\end{rem}

\section{The finite quotient singularity}
\label{sec:4}
 From now on we fix $n=5$. To prove Theorem~\ref{t:isom},  we first establish some information about the group $G$, expanding on \cite[\S 2.2]{BS-sra}
and \cite[\S 3.C]{81Coxres}. We then identify $\CC^4/G$ with the affine hyperpolygon space $X_5(0)$ by combining the analysis from \cite[\S 3.C, 3.D]{81Coxres} with our description of the semi-invariant ring from Proposition~\ref{p:gens-repsl2}.
This section is logically independent from the results of Section~\ref{sec:theresofhyperpoly}.

\subsection{Facts about $G$}
As in the introduction, set $V=\CC^4$ and consider  
$G = Q_8 \times_{\ZZ_2} D_8 < \Sp(V)$.
Recall that there is a two-to-one covering
$\Sp_4 \onto \SO_5$, which can be explicitly realised as follows.
The representation $\wedge^2 V^*$ decomposes as a sum of the
invariant part $\CC \cdot \omega$ (for $\omega$ the symplectic form)
and a five-dimensional irreducible representation $W$. The form $\omega^{-1} \otimes \omega^{-1}$ defines a nondegenerate symmetric bilinear form on $\wedge^2 V^*$ which restricts to one on $W$. We
then obtain a natural map 
\[
\gamma\colon \Sp(V) \to \SO(W)
\]
with kernel $\{\pm 1\}$, which is surjective for dimension reasons. This isogeny realises the equality of root systems $\mathsf{C}_2 = \mathsf{B}_2$. The centre of $G$ is $\ZZ_2 \times_{\ZZ_2} \ZZ_2 = \{\pm 1\}$, which is also equal to $[G,G]$. Thus, the image $\gamma(G)$ is isomorphic to the abelianisation $\Ab(G) \cong \ZZ_2^4$. This subgroup of $\SO(W)$ has a simultaneous orthogonal eigenbasis, so that it is identified under this choice of basis with the subgroup of diagonal matrices. 

\begin{prop}\label{p:G-facts}
	There is an orthonormal basis of $W$, unique up to scaling and reordering, such that $G$ is the preimage under $\gamma$ of the group of diagonal matrices $\ZZ_2^4 < \SO(W)$. The conjugation action of the permutation group $S_5 < \Ort(W)$
	lifts uniquely to a conjugation action $S_5 \to \Aut(G)$, whose composition with $\Aut(G) \onto \Out(G)$ is an isomorphism $S_5 \cong \Out(G)$. In particular, we have a semidirect product decomposition $\Aut(G) \cong \Inn(G) \rtimes \Out(G) \cong \ZZ_2^4 \rtimes S_5$.
\end{prop}
\begin{proof}
	The first statement follows from the spectral theorem for orthogonal matrices. 	For the second, we have 
	$\Ort(W) = \{\pm \Id\} \cdot \SO(W)$, giving a two-to-one
	covering $\mu_4 \cdot \Sp(V) \to \Ort(W)$. The preimage of $S_5$ in
	$\mu_4 \cdot \Sp(V)$, call it $\widetilde S_5$, acts on $G$ by
	conjugation, with kernel $\pm \Id$ (note that $\mu_4$ is not a
	subgroup of $\widetilde S_5$, since $\pm \Id$ is not a subgroup of
	$S_5$).  Therefore this action descends to the quotient $S_5$.
	Since these induce all (nontrivial) outer automorphisms of
	$\ZZ_2^4$, the composition $S_5 \to \Aut(G) \to \Out(G)$ is
	injective. Since $\Out(G) \cong S_5$ by \cite[Proposition
	2.2.2]{BS-sra}, this is actually an isomorphism, and the
	first map gives a homomorphic section of the quotient
	$\Aut(G) \to \Out(G)$. This establishes the last fact.
\end{proof}
\begin{remark}
	Since there is a unique irreducible representation of $G$ of
	dimension greater than one (see \cite[Proposition~2.0.1(v)]{BS-sra}), every subgroup of $\Sp_4$
	isomorphic to $G$ is actually conjugate to it.  Similarly, by the spectral
	theorem as in the beginning of the proof of Proposition~\ref{p:G-facts}, every
	subgroup of $\SO_5$ isomorphic to $\ZZ_2^4$ is conjugate to the
	diagonal subgroup.  So $G$ is also the unique subgroup of $\Sp_4$,
	up to conjugation, whose image in $\SO_5$ is isomorphic to
	$\ZZ_2^4$.
\end{remark}

\subsection{The quotient $V/G$}

As observed in \cite[\S 3.C]{81Coxres}, the coordinate ring of $V/G$ can be written as
\[
\CC[V]^G\cong\big(\CC[V]^{[G,G]}\big)^{\Ab(G)},
\]
where
$\CC[V]^{[G,G]} = \CC[V]^{\{\pm \Id\}}$
is the algebra
of even-degree polynomial functions on $V$; this is generated by $\Sym^2 V^*$. The $\Sp(V)$ action here descends under
the two-to-one cover to an action of $\SO(W)$.  As pointed out in
\emph{loc.~cit.}~, we have an isomorphism
\[
\iota\colon \Sym^2 V^* \longrightarrow \wedge^2 W
\]
of irreducible $\SO(W)$-representations that must be unique up to scaling.

\begin{remark}
\label{rem:iota-compute}
The isomorphism $\iota$ can be computed explicitly as follows: recall that $W \subseteq \wedge^2 V^*$ and hence $\wedge^2 W \subset \wedge^2 (\wedge^2 V^*)$.  We can contract with the symplectic form $\omega^{-1} \in \wedge^2 V$ and obtain an element of $V^* \otimes V^*$. Since $W$ is the kernel of the contraction $\wedge^2 V^* \to \CC$, the image of the contraction is symmetric, giving a map $\wedge^2 W \to \Sym^2 V^*$.  It is clear that this is compatible with the action of the symplectic group. Now $\iota$ is the inverse of this map (up to scaling).
\end{remark}

 Putting this together, the orthonormal basis $w_1,\ldots,w_5$ of $W$ as in Proposition~\ref{p:G-facts} (unique up to scaling and reordering)
gives $\CC$-algebra generators $\psi_{i,j} := \iota^{-1}(w_i \wedge w_j)$ of $\CC[V]^{[G,G]}$, where $1\leq i<j\leq 5$. 

\begin{prop}
\label{p:rels-qt}
	The ten elements $\psi_{i,j}$ in $\CC[V]^{[G,G]}$ satisfy the equations
	\begin{gather}
		\psi_{i,j} \psi_{k,\ell} - \psi_{i,k} \psi_{j,\ell} + \psi_{i,\ell} \psi_{j, k} = 0, \\
		\sum_i \psi_{i,j} \psi_{i,k} = 0.
	\end{gather}
	The subring $\CC[V]^G$ is spanned by products of generators in
	which every index occurs an even number of times.
\end{prop}

\begin{remark}
This result was established by \cite[Proposition 3.17]{81Coxres} and \cite[Proposition~3.1]{HausenKeicher} with some computer assistance, at least up to a choice of signs in the relations and in each $\psi_{i,j}$. 
\end{remark}

\begin{proof}
	The relations can be checked by an explicit computation using the description of $\iota$ from Remark~\ref{rem:iota-compute}; note by symmetry that it suffices to establish the first identity for a particular choice of distinct $i,j,k,\ell$ (the identity is trivial if any two indices are equal), and the second for a particular choice of $j \neq k$ and a particular choice of $j = k$.
	
For the convenience of the reader, let us explain how to derive the formulas for $\psi_{i,j}$.  First one must find the aforementioned orthonormal basis of $W$. 
Actually, any orthonormal basis will do,  since by conjugating $G$ by an element of $\Sp(V)$, the image $\ZZ_2^4$ of $G$ under $\Sp(V) \twoheadrightarrow \SO(W)$ is the group of diagonal matrices in any chosen orthonormal basis. To write one, let $v_i$ be a basis of $V$ and $x_i$ the dual basis of $V^*$. Let $\omega^{-1} =  v_1 \wedge v_3 + v_2 \wedge v_4$ on $V$.  This induces a nondegenerate symmetric bilinear form on $V^* \otimes V^*$, which restricts to one on $W$.
Up to overall scaling an orthonormal basis is, for $i := \sqrt{-1}$:
 $w_1=x_1 \wedge x_3 - x_2 \wedge x_4$, $w_2=x_1 \wedge x_2 + x_3 \wedge x_4$, $w_3 = i(x_1 \wedge x_2 - x_3 \wedge x_4)$, $w_4 = x_1 \wedge x_4 + x_2 \wedge x_3$, $w_5 = i(x_1 \wedge x_4 - x_2 \wedge x_3)$. Now the explicit map of Remark \ref{rem:iota-compute} immediately yields the following:
 \begin{center}
    \begin{tabular}{cclcccl}
$\psi_{1,2}$ & = &  $2(x_1 x_2 + x_3 x_4)$ & & $\psi_{2,4}$ & = & $x_1^2 - x_2^2 + x_3^2 - x_4^2$ \\
$\psi_{1,3}$ & = & $2i(x_1 x_2 - x_3 x_4)$ & & $\psi_{2,5}$ & = & $i(x_1^2 + x_2^2 - x_3^2 - x_4^2)$ \\
$\psi_{1,4}$ & = & $2(x_1 x_4 - x_2 x_3)$ & & $\psi_{3,4}$ & = & $i(x_1^2 - x_2^2 - x_3^2 + x_4^2)$ \\
$\psi_{1,5}$ & = & $2i(x_1 x_4 + x_2 x_3)$ & & $\psi_{3,5}$ & = & $-x_1^2 - x_2^2 - x_3^2 - x_4^2$ \\
$\psi_{2,3}$ & = & $-2i(x_1 x_3 + x_2 x_4)$ & & $\psi_{4,5}$ & = & $2i(-x_1 x_3 + x_2 x_4)$. 
\end{tabular}\end{center}
For example, $\psi_{1,2}$ is the contraction of $w_1 \wedge w_2$ with $\omega$: 
 sum over all ways of contracting the symplectic form $\omega$ to one component of $w_1$ and one of $w_2$ and multiply the other two components with the correct sign, yielding $x_1 x_2 +x_3 x_4 + x_2 x_1 + x_4 x_3 = 2(x_1x_2 + x_3 x_4)$. (Note that the above table corrects \cite[(3.13)]{81Coxres}, which almost matches ours after reordering and applying an overall scaling, but some coefficients in \emph{op.~cit.} are incorrect.)
 
	The algebra $\CC[V]^G$ is then identified with the subalgebra
	of $\CC[V]^{[G,G]}$ generated by those products of weight vectors (for $\Z_2^4$) in $\Sym^2 V^* \cong \wedge^2 W$ for which the weights sum to
	zero. Explicitly, we may identify $\Hom(\ZZ_2^4, \CC^{\times})$ with
	$\ZZ_2^4$ so that the weights on $W$
	are $e_1, e_2, e_3, e_4, -e_1-e_2-e_3-e_4$, with $e_i \in \ZZ_2^4$
	the standard coordinate vectors (we have
		$-e_1-e_2-e_3-e_4=e_1+e_2+e_3+e_4$, but writing in the above way
		resembles the reflection representation of $\SL_5$).  Then the
	weights on $\wedge^2 W \cong \Sym^2 V^*$ are
	$e_i + e_j, e_i - e_1-e_2-e_3-e_4$ for $i < j$.  Therefore, the
	products of generators $\psi_{i,j}$ having trivial weights are
	those in which the parity of the number of occurrences of every index $i$ equals the parity of the number of occurrences of the index $5$. But since there must be a total even number of indices occurring, this means that actually every index must occur an even number of times.
\end{proof}

\subsection{The isomorphism}
We are now in a position to prove Theorem~\ref{t:isom}.

\begin{thm} 
\label{thm:4foldsing}
The assignment $\varphi_{a_i, a_j}, \varphi_{a_i, a_j^*}, \varphi_{a_i^*, a_j}, \varphi_{a_i^*, a_j^*} \mapsto \psi_{i,j}$ extends to a surjective $\CC$-algebra homomorphism $\CC[\mu^{-1}(0)]^{\SL_2} \to \CC[V]^{[G,G]}$ that restricts to an isomorphism  $\CC[\mu^{-1}(0)]^{\GL_{\bv}} \to \CC[V]^G$. In particular, $X_5(0)\cong \mathfrak{M}_0(\bv,0)\cong V/G$.
\end{thm}
\begin{proof} 
Since $\lambda=0$, Proposition~\ref{p:gens-repsl2} says that the equations
		\begin{equation}
		\varphi_{a,b} \varphi_{c,d} - \varphi_{a,c} \varphi_{b,d} + \varphi_{a,d} \varphi_{b,c} = 0
		\end{equation}
		\begin{equation}
		\sum_i \varphi_{a_i,b} \varphi_{a_i^*, c} =0  
		\end{equation}
		hold in $\CC[\mu^{-1}(0)]^{\SL_2}$, so the surjective $\CC$-algebra homomorphism is well-defined by Proposition~\ref{p:rels-qt}. The restriction of this map to $\CC[\mu^{-1}(0)]^{\GL_{\bv}}$ is well-defined and maps onto $\CC[V]^G$ by these same propositions, giving a closed immersion $V/G \into \Mq_{0}(\bv,0)$. But these are irreducible varieties of the same dimension, so this closed immersion is an isomorphism.
\end{proof}

\section{Affine Hyperpolygon spaces are not quotient singularities more generally} \label{sec:5}
We have seen that the hyperpolygon space $\Hyp_n(0)$ is a finite symplectic quotient singularity for $n = 4$ or $5$. It is therefore natural to ask if we can identify $\Hyp_n(0)$ for $n > 5$ with a symplectic quotient singularity. In this section, we establish the following result. 

\begin{thm}\label{t:not-quotient-sing}
For $n > 5$, the affine hyperpolygon space $\Hyp_n(0)$ is not isomorphic to $V/G$ for any finite group $G \subset \Sp(V)$. 
\end{thm}

\subsection{The movable cone of a product}
The proof of Theorem~\ref{t:not-quotient-sing} requires preparatory results. The following is presumably well-known, but we were unable to find the statement in the literature; for the statement under the additional assumption that $X_1$ is rational, see \cite[Lemme~6.6]{Sansuc}. 

\begin{lem}\label{lem:productpicfgnew}
Let $X_1,X_2$ be smooth (connected) quasi-projective varieties with $\Pic(X_1)$ finitely generated. Then, $\Pic (X_1 \times X_2) \cong \Pic(X_1) \oplus \Pic(X_2)$. 
\end{lem}
\begin{proof}
We choose a projective compactification $\overline{X}_i$ of $X_i$. Since any quasi-projective variety over $\C$ admits a projective resolution of singularities by Hironaka's Theorem, we may assume that each $\overline{X}_i$ is smooth. If $E_{1}, \ds, E_k$ (resp. $F_1, \ds, F_{\ell}$) are the irreducible components of codimension one in $\overline{X}_1 \smallsetminus X_1$ (resp. in $\overline{X}_2 \smallsetminus X_2$) 
then standard exact sequence on divisor class groups gives 
$$
\bigoplus_{i = 1}^k \Z [\mc{O}(E_i)] \longrightarrow \Pic(\overline{X}_1) \longrightarrow \Pic(X_1) \longrightarrow 0, 
$$
and similarly for $X_2$. 
In particular, $\Pic(\overline{X}_1)$ is finitely generated. Since $\overline{X}_1$ is projective over $\C$, we may consider $\Pic(\overline{X}_1)$ as a group scheme; see \cite[Theorem~4.8]{KleimanSurvey}. Since it is a discrete group, the connected component $\Pic^0(\overline{X}_1)$ of the identity is trivial.

Now assume $\mc{M}$ is a line bundle on $\overline{X}_1 \times \overline{X}_2$ and there exists $y_0 \in \overline{X}_2$ with $\mc{M}_{y_0}$ (the restriction of $\mc{M}$ to $\overline{X}_1 \times \{ x_2 \}$) trivial. Then \cite[Proposition~5.10]{KleimanSurvey}  implies that every $\mc{M}_{y}$, for $y \in \overline{X}_2$, is trivial since $\overline{X}_2$ is connected. We deduce, more generally, that if $\mc{L}$ is a line bundle on $\overline{X}_1 \times \overline{X}_2$ then $\mc{L}_{y_1} \cong \mc{L}_{y_2}$ for all $y_i \in \overline{X}_2$. Set $\mc{K} = \mc{L}_{y_1}$ and consider the projections $\pi_i \colon \overline{X}_1 \times \overline{X}_2 \to \overline{X}_i$. Since $\pi_2$ is a flat projective morphism, \cite[III, Exercise~12.4]{Hartshorne} says that there exists a line bundle $\mc{N}$ on $\overline{X}_2$ such that $\mc{L} \cong \pi_1^* \mc{K} \otimes \pi^*_2 \mc{N}$. That is, $\mc{L} \cong \mc{K} \boxtimes \mc{N}$. Thus, the canonical map $\Pic(\overline{X}_1) \oplus \Pic(\overline{X}_2) \to \Pic(\overline{X}_1 \times \overline{X}_2)$ is surjective. The fact that it is injective is immediate: if $\mc{K} \boxtimes \mc{N} \cong \mc{O}$ then pulling back along $i \colon \overline{X}_1 \times \{ y_0 \} \to \overline{X}_1 \times \overline{X}_2$ implies that $\mc{K} \cong \mc{O}$ and similarly for $\mc{N}$. We deduce that $\Pic(\overline{X}_1) \oplus \Pic(\overline{X}_2) \cong \Pic(\overline{X}_1 \times \overline{X}_2)$. 

Finally, the corresponding result for $X_1 \times X_2$ follows from the fact that the rows in the commutative diagram
$$
\begin{tikzcd}
\bigoplus_{i = 1}^k \Z [\mc{O}(E_i)] \oplus \bigoplus_{j = 1}^{\ell} \Z [\mc{O}(F_j)] \ar[r] \ar[d,equal] & \Pic(\overline{X}_1 \times \overline{X}_2) \ar[r] & \Pic(X_1 \times X_2) \ar[r] & 0 \\
\bigoplus_{i = 1}^k \Z [\mc{O}(E_i)] \oplus \bigoplus_{j = 1}^{\ell} \Z [\mc{O}(F_j)] \ar[r] & \Pic(\overline{X}_1) \oplus \Pic(\overline{X}_2) \ar[r] \ar[u,"\wr"'] & \Pic(X_1) \oplus \Pic(X_2) \ar[r] \ar[u] & 0, 
\end{tikzcd}
$$
are exact. 
\end{proof}

 Let $Y$ be a conical symplectic singularity (see Section \ref{sec:sympleaves}). Suppose there exists a projective, crepant resolution of singularities $f\colon X \to Y$. Recall from Hu--Keel~\cite{HuKeel00} (see Grab~\cite{GrabThesis} and Ohta~\cite{Ohta} for the relative case) that $X$ is a \emph{relative Mori Dream Space over} $Y$ if \begin{enumerate}
 \item[\one] $\Pic(X/Y)_{\mathbb{Q}}\cong N^1(X/Y);$
 \item[\two] the relative nef cone $\Nef(X/Y)$ is generated by finitely many semiample
line bundles; and
 \item[\three] there exists $k\geq 0$ and  $\mathbb{Q}$-factorial varieties $X=X_0, X_1, \dots, X_k$, each projective over $Y$, as well as birational maps $\psi_i\colon X\dashrightarrow X_i$ over $Y$ for $0\leq i\leq k$ that are isomorphisms in codimension-one, such that  
 \begin{equation}
 \label{eqn:movconedecomp}
\Mov(X/Y) = \bigcup_{0\leq i\leq k} \psi_i^*\;\Nef(X_i/Y),
\end{equation}
 where each cone in this description is generated by finitely many semiample line bundles.
 \end{enumerate}
 \noindent In particular, the decomposition \eqref{eqn:movconedecomp} describes $\Mov(X/Y)$ as a finite polyhedral fan. The following result was first observed in the current context by Andreatta--Wi\'{s}niewski~\cite{AndreattaWisniewski14} in dimension four, and by Namikawa~\cite{NamikawaMDS} in general.
 
\begin{lem}
\label{lem:BCHM}
 The resolution $X$ is a relative Mori Dream Space over $Y$. In particular, the polyhedral decomposition \eqref{eqn:movconedecomp} depends only on $Y$, not on the choice of the projective crepant resolution.
\end{lem}
\begin{proof}
Observe that any divisor $D$ on $X$ is effective: the coherent sheaf $f_*(\mathcal{O}_X(D))$ has nonvanishing global sections as $Y$ is affine and $f$ is birational. In particular, $-A$ is effective for any $f$-ample divisor $A$. The same is true for $\Delta:=-\epsilon A$ for any rational number $\epsilon > 0$. Letting $\epsilon$ be sufficiently small, we see that $(X,\Delta)$ is klt, so it is divisorially log terminal by definition (cf.~\cite[Proposition~2.41]{KollarMori}). Since $-\Delta$ is $f$-ample by construction,  \cite[Corollary~1.3.2]{BCHM} gives that the Cox ring of $X$ is finitely generated which, by \cite[Corollary~1.1.5]{BCHM} (see also Grab~\cite[Theorem~3.4.7]{GrabThesis} or Ohta~\cite[Corollary~6.14]{Ohta}) implies that $X$ is a relative Mori Dream Space over $Y$. 

Finally, let $X\to Y$ and $X^\prime\to Y$ be projective crepant resolutions and let $\psi \colon X \dashrightarrow X^\prime$ be the resulting rational map over $Y$. Then $\psi$ is an isomorphism in codimension one \cite[Corollary 3.54]{KollarMori}, so we may identify the relative N\'{e}ron--Severi spaces of $X$ and $X^\prime$ over $Y$ via strict transform along $\psi$. This, in turn, identifies $\mr{Mov}(X/Y)$ and $\mr{Mov}(X^\prime/Y)$ as fans; see \cite[Definition~1.10(3)]{HuKeel00}. 
\end{proof}

\begin{remark}
\label{rem:BCHM}
The proof of Lemma~\ref{lem:BCHM} relies on \cite{BCHM}, because in order to understand the decomposition of $\Mov(X/Y)$ beyond the nef cone cone $X$ over $Y$, one must know that flops exist in arbitrary dimension. For projective resolutions constructed by VGIT, an alternative approach is given in \cite{BCS}.
\end{remark}

Namikawa~\cite{NamikawaMDS} proves that the fan decomposition  from Lemma~\ref{lem:BCHM} is given by intersecting the movable cone with  the complement of a hyperplane arrangement $\mc{D}_Y \subset N^1(X/Y)$. By Lemma~\ref{lem:BCHM}, the fan decomposition depends only on $Y$ and hence so does the hyperplane arrangement. Namikawa introduces the action of a group $W_Y$ on $N^1(X/Y)$, depending only on  $Y$ \cite{Namikawa} and, moreover, the movable cone is a fundamental domain for this action \cite{BLPWAst}. 

\begin{defn} 
We call $\mc{D}_Y$ the Namikawa hyperplane arrangement of $Y$ and $W_Y$ the Namikawa Weyl group of $Y$.
\end{defn}

It follows that in the case $Y=X_n(0)$, the arrangement $\mc{D}_Y$ equals $\mathcal{A}$ and the Namikawa Weyl group is $\mathbb{Z}_2^n$ as we found earlier.

 Now let $Y_1,Y_2$ be conical symplectic singularities. We assume that there exist projective crepant resolutions $f_i\colon X_i \to Y_i$ for $i=1, 2$. Define $X:= X_1 \times X_2$ and $Y:= Y_1 \times Y_2$, and consider the projective, crepant resolution $f=f_1\times f_2\colon X \to Y$. In particular, we may apply Lemma~\ref{lem:BCHM} to each of $X_1, X_2$ and $X$ over $Y_1, Y_2$ and $Y$ respectively.

\begin{prop}\label{prop:movableproduct}
The Picard group of $X_i$ is finitely generated for $i=1,2$, and moreover, the isomorphism of Lemma~\ref{lem:productpicfgnew} induces an isomorphism of relative N\'{e}ron--Severi spaces 
\begin{equation}
\label{eqn:N1product}
N^1(X/Y)\cong N^1(X_1/Y_1)\times N^1(X_2/Y_2).
\end{equation}
 This, in turn, induces an identification of finite polyhedral fans
\begin{equation}
\label{eqn:Movproduct}
\mr{Mov}(X/Y)\cong \mr{Mov}(X_1/Y_1)\times \mr{Mov}(X_2/Y_2).
\end{equation}
In particular, the Namikawa hyperplane arrangement of $Y$ is equal to the product of the Namikawa hyperplane arrangements of $Y_1$ and $Y_2$.
\end{prop}
 
\begin{proof}
 To prove that $\Pic(X_i)$ for $i=1,2$ is finitely generated, it suffices to prove that the $\ZZ$-linear map $\Pic(X_i)\to N^1(X_i/Y_i)$ sending a line bundle to its numerical class is injective. Let $L\in \Pic(X_i)$ satisfy $\deg(L\vert_C)=0$ for each curve $C$ contracted by $f_i$. Consider the pair $(X_i,\Delta)$ constructed in the proof of Lemma~\ref{lem:BCHM}, where $-\Delta$ is $f_i$-ample. Since $K_{Y_i} = 0$ and $f_i$ is crepant, we have $K_{X_i} = 0$ and hence $-(K_{X_i}+\Delta) = - \Delta$ is $f_i$-ample. The relative version of Kleiman's ampleness criterion now gives $(K_{X_i}+\Delta).z<0$ for all $z\in \overline{NE}(X_i/Y_i)\setminus \{0\}$, in which case the relative version of the Contraction Theorem~\cite[Theorem~3-2-1]{KawamataMatsudaMatsuki} gives that $L\cong (f_i)^*(M)$ for some ample line bundle $M$ on $Y_i$. However, the Picard group of $Y_i$ is trivial because $Y_i$ is a conical singularity, so $L\cong \mathcal{O}_{X_i}$. This shows that the map $\Pic(X_i)\to N^1(X_i/Y_i)$ is injective as required.
 
  We may now apply Lemma~\ref{lem:productpicfgnew} to obtain $\Pic (X) \cong \Pic(X_1) \oplus \Pic(X_2)$. For the  statement \eqref{eqn:N1product}, it suffices to show that if a line bundle $L$ has degree zero on all proper curves in $X_1$ and $X_2$, then it also has degree zero on every proper (connected) curve $C$ in $X_1 \times X_2$. But the intersection number $L.C = \deg (L\vert_C)$ depends only on the class of $C$ in $H_2(X,\Q) = H_2(X_1,\Q) \oplus H_2(X_2,\Q)$. Since this class can be written in terms of classes of curves in $X_1$ and $X_2$, we obtain $\deg (L\vert_C) = 0$.

For the statement \eqref{eqn:Movproduct}, recall that any relative Mori Dream Space $X$ over $Y$ can be constructed by GIT from its Cox ring $\mathcal{R}(X)$. The Cox ring of $X$ is the tensor product $\mathcal{R}(X)\cong \mathcal{R}(X_1)\otimes_{\mathbb{C}} \mathcal{R}(X_2)$ of the Cox rings of $X_1$ and $X_2$, equipped with the canonical grading by $\Pic(X_1)\oplus \Pic(X_2)$; see \cite[Lemma~4.2.2.3]{ADHL15}.  In particular, $X=X_1\times X_2$ is a GIT quotient for the action of the quasitorus $\Spec \CC[\Pic(X_1)]\times \Spec \CC[\Pic(X_2)]$ on $\Spec \mathcal{R}(X_1)\times \Spec \mathcal{R}(X_2)$ determined by the product group action, so the GIT chamber decomposition for the $\Pic(X)$-grading of $\mathcal{R}(X)$ is the product of the GIT chamber decompositions for the $\Pic(X_i)$-gradings of $\mathcal{R}(X_i)$ for $i=1,2$.

Within this product chamber decomposition, we need to check that the chambers whose union makes up $\mr{Mov}(X/Y)$ is precisely the the union of the product of chambers making up $\mr{Mov}(X_1/Y_1)$ and $\mr{Mov}(X_2/Y_2)$. For all  three of these GIT constructions, the movable cone is a subfan of the GIT fan characterised by the $\Pic$-degrees of a system of homogeneous generators of the relevant Cox ring~\cite[Proposition~3.3.2.3]{ADHL15}. A system of homogeneous generators of $\mathcal{R}(X)$ is given by taking the product $\{ g_i h_j \mid i \in I, j \in J \}$ over systems of homogeneous generators $\{ g_i \mid i \in I\}$ and $\{ h_j \mid j \in J \}$ of $\mathcal{R}(X_1)$ and $\mathcal{R}(X_2)$ respectively. It follows that the cone supporting the movable cone of $X_1\times X_2$ over $Y_1\times Y_2$ coincides with the product of the cones supporting the movable cones of $X_1$ over $Y_1$ and $X_2$ over $Y_2$ respectively. 

 Finally,  the movable cone contains the ample cone which is of full dimension for a projective morphism over an affine base. The statement about the Namikawa hyperplane arrangements now follows from the statement about the movable cones. This completes the proof.
\end{proof}

\subsection{Resolutions of quotient singularities}
Below, $V_1,V_2$ and $V$ will denote finite-dimensional symplectic vector spaces.

\begin{lem}\label{lem:nontrivaiNamarragement}
Let $G \subset \Sp(V)$ be a non-trivial finite group such that $Y = V/G$ admits a projective crepant resolution $f \colon X \to Y$. The Namikawa hyperplane arrangement in $N^1(X/Y)$ is non-trivial. 
\end{lem}
Observe in this case that $Y$ is singular by results of Shephard--Todd~\cite{ST} and Chevelley~\cite{Chevalley} because $G$ is not a complex reflection group; indeed it contains contains no complex reflections.
\begin{proof}
Since $V/G$ is singular, $f$ is not an isomorphism. Let $\mathcal{O}_X(D)$ be $f$-ample. The divisor $D$ is numerically positive against all (necessarily proper) curves contracted by $f$. In particular, $[D] \neq 0$ in $N^1(X/Y)$, so $N^1(X/Y) \neq 0$. Since $f$ is not an isomorphism, $0 \in N^1(X/Y)$ cannot be $f$-ample. Thus, there is a hyperplane containing $0$; in fact, the Namikawa arrangement is central so it lies on all hyperplanes.
\end{proof}

 A hyperplane arrangement is reducible if it is a product of non-trivial hyperplane arrangements. More formally, $\mathcal{A}$ is \emph{reducible} if  if there exist nonempty subarrangements $\mathcal{A}_1$ and $\mathcal{A}_2$ such that $\mathcal{A}=\mathcal{A}_1\cup \mathcal{A}_2$ is a disjoint union, and after a linear change of coordinates, the linear polynomials defining the hyperplanes in $\mathcal{A}_1$ and in $\mathcal{A}_2$ share no common variable.

\begin{lem}\label{lem:GproductNamarrangement}
Let $G_i \subset \Sp(V_i)$ for $i = 1,2$ be finite groups, both non-trivial. If $Y = V_1 / G_1 \times V_2 / G_2$ admits a projective crepant resolution $X^\prime \to Y$, then the Namikawa hyperplane arrangement in $N^1(X^\prime/Y)$ is reducible.  
\end{lem}

\begin{proof}
Applying \cite[Theorem~1.6]{KaledinDynkin} to the action of $G=G_1\times G_2$ on $V=V_1\times V_2$ implies that each $Y_i = V_i/G_i$ admits a projective crepant resolution $X_i \to Y_i$. Thus, $X_1 \times X_2 \to Y$ is also a projective crepant resolution of $Y$. As described in the final paragraph of the proof of Lemma~\ref{lem:BCHM}, we may identify $\Mov(X^\prime/Y)$ with $\Mov(X_1\times X_2/Y)$. The result follows from Proposition~\ref{prop:movableproduct}.   
\end{proof}

\begin{lem}\label{lem:hyperirreducible}
The hyperplane arrangement $\mc{A}$ in $\Theta$ is irreducible.  
\end{lem}

\begin{proof}
Assume that $\Theta = \Theta_1 \times \Theta_2$ and $\mc{A} = \mc{A}_1 \cup \mc{A}_2$ for $\mc{A}_i$ an arrangement in $\Theta_i$. Then, for each $H \in \mc{A}$, $H \in \mc{A}_1$ means that $\Theta_2 \subset H$ etc. Since the coordinate hyperplanes $L_i$ belong to $\mc{A}$, this forces $\Theta_1 = (\theta_i = 0 : i \notin J)$ and $\Theta_2 = (\theta_j = 0 : j \in J)$ for some $J \subset [1,n]$. But then the $H_I$ do not contain either $\Theta_1$ or $\Theta_2$ unless $J \neq \emptyset, [1,n]$. Thus, $\mc{A}$ is irreducible.  
\end{proof}

If $G \subset \Sp(V)$ is a finite group then the pair $(V,G)$ is said to be \textit{symplectically reducible} if $V = V_1 \times V_2$ is a proper decomposition into symplectic $G$-submodules. Otherwise, $(V,G)$ is said to be \textit{symplectically irreducible}; see \cite[Introduction]{BS-sra2} for more on this notion. Recall that $g\in G$ is a \emph{symplectic reflection} if it fixes a linear subspace of codimension two, and $G$ is a \emph{symplectic reflection group if it is generated by symplectice reflections.}

\begin{cor}\label{cor:hyperpolyquotiso}
If, for some $n \ge 4$, there is an isomorphism $\Hyp_n(0) \cong V/G$ with $G \subset \Sp(V)$ finite, then $(V,G)$ is a symplectically irreducible symplectic reflection group.  
\end{cor}

\begin{proof}
Choose $\theta \in \Theta$ generic. By Corollary~\ref{cor:symplecticreshyperpoly}, the morphism $X = \Hyp_n(\theta) \to \Hyp_n(0) = Y$ is a projective crepant resolution. Moreover, under the identification $L_F \colon \Theta \iso N^1(X/Y)$ of Theorem~\ref{thm:resolutionsofstar}, the arrangement $\mc{A}$ in $\Theta$ is identified with the Namikawa arrangement. Therefore, Lemma~\ref{lem:hyperirreducible} implies that the latter is irreducible.

If $X \cong V/G$, then $G$ is a symplectic reflection group by Verbitsky's Theorem \cite[Theorem~3.2]{Verbitsky} because $X$ admits a projective crepant resolution. In this case, if $V$ is symplectically reducible then $V = V_1 \times V_2$ and $G = G_1 \times G_2$, with $G_i$ acting on $V_i$. Moreover, $G_i \neq 1$ because $X_n(0)$ has a zero-dimensional leaf by Lemma~\ref{lem:sympleavesH0} and symplectic leaves of a symplectic singularity are preserved under isomorphism (they are the strata of the singular locus stratification \cite{Kaledinsympsingularities}). Lemma~\ref{lem:GproductNamarrangement} then implies that the Namikawa arrangement in $N^1(X/Y)$ is reducible. This contradicts the conclusion of the previous paragraph. Thus, $(V,G)$ is symplectically irreducible. 
\end{proof}

\begin{proof}[Proof of Theorem~\ref{t:not-quotient-sing}]
Assume that $\Hyp_n(0) \cong V/G$. By Corollary~\ref{cor:hyperpolyquotiso}, $(V,G)$ is symplectically irreducible and $V/G$ admits a projective crepant resolution. We also know from Proposition~\ref{prop:beyondF} that its associated Namikawa Weyl group is $\Z_2^n$. Based on the the paragraph after Question~9.5 of \cite{BS-sra2} and \cite[Remark~6.2]{BellSchmittThiel}, either $G = \s_n \wr K$ is a wreath product group (with $K \subset \SL(2)$) or potentially $G$ is one of the (complex) primitive symplectic reflection groups $W(S_1), W(R)$ or $W(U)$ acting on a $6,8$ or $10$ dimensional space, respectively (in the notation of \cite[Section~4]{CohenQuaternionic}).

If $G = \s_n \wr K$ for $n>1$ and $W_K$ is the (irreducible) Weyl group of type ADE associated to $K$ via the McKay correspondence, then by \cite[Theorem~1.7]{BC} the Namikawa Weyl group of a projective symplectic resolution $X$ of $V / G$ is $\Z_2 \times W_K$. Then $\Z_2 \times W_K \cong \Z_2^n$, but the only abelian Weyl group of type $ADE$ is of type $A_1$, where $W_K\cong \Z_2$ and hence $n = 2$. But we assumed $n > 5$. 

Therefore, we must rule out $G = W(S_1), W(R)$ or $W(U)$. It was shown recently by computer computations in \cite{SympSteinberg} that these quotient singularities do not admit projective symplectic resolutions, but to avoid relying on computer calculations, we do not use this result here. To rule these cases out directly, note that if $\Hyp_n(0)$ is isomorphic to $V/G$ then $\Hyp_n(\theta) \to V/G$ is a projective symplectic resolution for $\theta$ generic. On the one hand, as in the proof of Proposition~\ref{prop:dimN1}, $\dim H^2(\Hyp_n(\theta),\C) = n$. On the other, by Lemma~\ref{lem:H2Wexceptional} below, $\dim H^2(\Hyp_n(\theta),\C) = 1$; a contradiction.  
\end{proof}

In the statement, and proof, of the following lemma, we use freely the notation from \cite[Section~4]{CohenQuaternionic} where the irreducible symplectic reflection groups with primitive complexification are classified.  

\begin{lem}\label{lem:H2Wexceptional} 
Let $G$ equal one of the (complex) primitive symplectic reflection groups $W(S_1), W(R)$ or $W(U)$ and assume that $V/G$ admits a projective symplectic resolution $X \to V/G$. Then $H^2(X,\Q)$ is one-dimensional. 
\end{lem}

\begin{proof}
By \cite[Corollary~1.5]{ItoReid}, $\dim H^2(X,\Q)$ equals the number of ``junior'' conjugacy classes in $G$ which, in light of \cite[Lemma~2.6]{KaledinMcKay}, can be defined to be the conjugacy classes of symplectic reflections. As noted in \cite[Remark~4.3(ii)]{CohenQuaternionic}, each symplectic reflection in $G$ (for our choice of groups) is of the form $s_{a,\xi}$ for some $(a,\xi) \in \Sigma$, where $\Sigma$ is the root system associated to $G$, as in \cite[Table~II]{CohenQuaternionic}. Then, 
$$
\xi \in H_a := \{ \lambda \in SU(2) \, | \, s_{a,\lambda} \in G \}.
$$
Inspecting \cite[Table~II]{CohenQuaternionic} shows that $H_a = \Z_2$, and hence $\xi = -1$ for our three groups. It is explained in the proof of \cite[Theorem~4.2]{CohenQuaternionic} that the roots $a$ in $\Sigma$ form a single $G$-orbit. We deduce that all symplectic reflections are conjugate.    
\end{proof}

\section{Hyperk\"ahler metrics}
\label{sec:6}
Recall that $X_n(\theta)$ carries a complete hyperk\"ahler metric for any generic choice of $\theta$. Recently, there has been much interest in the asymptotic geometry of such spaces. In this final (discursive) section, we note the implications of our main results for asymptotic geometry of $X_n(\theta)$. The resolution of singularities $X_n(\theta) \to X_n(0)$ identifies the singularity $X_n(0)$ with the tangent cone at infinity of the hyperk\"ahler manifold $X_n(\theta)$. The full description of the asymptotic geometry of $X_n(\theta)$ requires precise estimates on the decay of the true hyperk\"ahler metric on $X_n(\theta)$ to the (singular) cone metric on $X_n(0)$. 

\subsection{Metrics: the case $n = 4$}

It is known that $X_{4}(\theta)$ is \emph{asymptotically locally Euclidean} (ALE). In general, for a smooth, noncompact hyperk\"ahler manifold of real dimension $4k$, this means that the asymptotic geometry is modelled on the singular space $\C^{2k}/G$ for some finite subgroup $G$ in $\Sp(k) = \Sp(2k,\C) \cap \mathrm{U}(2k)$, whose action on the unit sphere $S^{4k-1} \subset \C^{2k}$ is free. In particular, the orbifold cone $\C^2/G$ has an isolated singularity at the origin. For $k=2$, the hyperk\"ahler ALE-spaces were classified by Kronheimer~\cite{Kronheimer1, Kronheimer2} to be the crepant resolutions of the Kleinian singularities $\C^2/G$ for the finite subgroups $G < \mathrm{SU}(2)=\Sp(1)$. The hyperpolygon space $X_{4}(\theta)$ corresponds to the (affine) $D_4$ case, that is, to the quaternion group $Q_8 < \mathrm{SU}(2)$, and hence $X_4(0)$ is the orbifold cone $\C^2/Q_8$ which has an isolated singularity at the origin.

\subsection{Metrics: the case $n = 5$}

As a further generalisation of ALE, a smooth noncompact hyperk\"ahler manifold is  called \emph{quasi-asymptotically locally Euclidean} (QALE) if its asymptotic geometry is modelled on $\C^{2k}/G$ for a finite subgroup $G$ in $\Sp(k)$ that does not necessarily act freely on the unit sphere in $\C^{2k}$; here $\C^{2k}/G$ has singular strata given by non-isolated quotient singularities. Both ALE and QALE metrics have maximal, that is, Euclidean volume growth. Regarding larger values of $n$, Theorem~\ref{t:isom} shows that the hyperpolygon space $X_{5}(0)$ is the orbifold cone $\CC^4/G$ for the group $G = Q_8 \times_{\ZZ_2} D_8$ in $\Sp(2)=\Sp(4,\CC) \cap \mathrm{U}(4)$. This provides significant evidence for the following conjecture:

\begin{conjecture}\label{conj:QALE}
    For generic $\theta$, the space $X_{5}(\theta)$ is quasi-asymptotically locally Euclidean.
\end{conjecture}

There are two reasons that make this conjecture plausible. First, by work of Joyce \cite{Joyce}, there exists a unique hyperk\"ahler QALE metric on the crepant resolution of a symplectic quotient singularity in each QALE K\"ahler class (with regards to the complex structure $I$). However, it is \emph{a priori} not at all clear that the natural metric on $X_5(\theta)$---that is, the one given by hyperk\"ahler reduction \cite{HKLR}---can also be obtained by the construction of Joyce. To apply Joyce's result one would first have to construct a QALE metric in the given K\"ahler class. While this may presumably be carried out by an inductive procedure, applying the uniqueness portion of Joyce's theorem would require specific knowledge about the asymptotics of the competing hyperk\"ahler metric; see \cite{Carron} for a discussion in the context of hyperk\"ahler metrics for the planar Hilbert scheme of points. 

Second, one may compare the metric geometry of the moduli space of (parabolic) Higgs bundles on an $n$-punctured $\mathbb{P}^1$ to that of $X_{n}(\theta)$, given the embedding of the former into the latter as described in \cite{GodinhoMandini,FisherRayan,RayanSchaposnik}.  In particular, one of the moment map conditions that defines $X_{n}(\theta)$ as a hyperk\"ahler quotient is an explicit linearisation of the relation that defines the character variety of an $n$-punctured sphere. The Higgs bundle moduli space, which is one of the avatars of the character variety under nonabelian Hodge theory, comes with certain expectations about the decay of its hyperk\"ahler metric to a certain ``semiflat'' model metric on a torus bundle over a half-dimensional base as one goes to infinity. This behaviour is broadly classed as ``ALG''  --- see, for example, \cite{Perspectives}.  We note that ALG is not an acronym; rather, it is obtained by iteration from ALE, with ALF as the intermediary.  In general, this entails a much slower volume growth. (In complex dimension $2$, the volume growth for ALG is quadratic and is explicitly verified for the $n=4$ parabolic Higgs moduli space in \cite{ALG}.)  The realisation of $X_{n}(\theta)$ as a linear version of this geometry suggests, for arbitrary $n$, a more rapid, namely fully Euclidean volume growth in contrast to its Higgs bundle counterpart.  For more on this point of view, see \cite{RayanSchaposnik}.

\subsection{Metrics: the general case}

For $n \geq 6$,  Theorem~\ref{t:not-quotient-sing} shows that $X_n(0)$ is not an orbifold cone of the form $\C^{2k}/G$ for a finite subgroup $G$ in $\Sp(k)=\Sp(2k),\C) \cap \mathrm{U}(2k)$, where $k = n-3$. As a consequence, we obtain a proof of Corollary~\ref{cor:QALEnot} which we restate here for convenience:

\begin{cor}
\label{cor:restated1.6}
For $n > 5$, $X_n(\theta)$ is not quasi-asymptotically locally Euclidean for any $\theta$.  
\end{cor}
\begin{proof}
  If the hyperk\"ahler metric on $X_n(\theta)$ given by hyperk\"ahler reduction were QALE, then the tangent cone at infinity $X_n(0)$---that is, \ the limit of the sequence of homothetic rescalings $(X_n(\theta),\lambda^2g)$ as $\lambda \to 0$---would be the Riemannian orbifold $\C^{2k}/G$. This contradicts Theorem~\ref{t:not-quotient-sing}.
\end{proof}

It is natural to ask whether $X_n(\theta)$ is \emph{quasi-asymptotically conical} (QAC) in the sense of Degeratu--Mazzeo~\cite{DegeratuMazzeo}. This notion generalises QALE in that it does not require the tangent cone at infinity to be an orbifold, while maintaining full Euclidean volume growth.

\subsection{The McKay correspondence.}
Finally, it is appropriate to frame this discussion within the context of the McKay correspondence.  The fact that hyperk\"ahler ALE metrics on smooth noncompact four-manifolds are determined, up to the choice of Torelli periods \cite{Kronheimer1,Kronheimer2}, by finite subgroups of $\mathrm{SU}(2)$, is a geometric extension of the usual McKay correspondence between finite subgroups of $\mathrm{SU}(2)$ and affine Dynkin diagrams of ADE type.  This geometric extension can be regarded as an ALE/ADE correspondence.  The variety $X_4(\theta)$ is precisely the affine $D_4$ case of this correspondence, with $X_4(\theta)$ and its hyperk\"ahler metric being recovered from the Dynkin quiver via the Nakajima quiver variety construction \cite{Nak1994}. While there are many versions of the McKay correspondence for finite subgroups of $\mathrm{SU}(3)$ in \cite{ReidMckay}, it is very rare that such results hold for finite subgroups of $\mathrm{SU}(r)$ for $r>3$. In this context, Theorem~\ref{t:isom} provides an instance where the quiver \emph{does} determine a finite subgroup of $\mathrm{SU}(4)$, and a corresponding hyperk\"ahler geometry. Note however that in our case, the determinants of all bar one of the tautological bundles associated to the vertices of the quiver in Figure~\ref{fig:quivers} provide a basis for the Picard group, but these bundles do not provide a basis of the Grothendieck group; in particular, they do not generate the derived category of coherent sheaves as is the case in the McKay correspondence.

\small{
\def\cprime{$'$}

\bibliographystyle{plain}

}

\end{document}